\documentclass[11pt]{amsart}
\usepackage{amsmath,amscd,amssymb}
\usepackage[mathscr]{eucal}
\usepackage[all]{xy}
\allowdisplaybreaks
\newtheorem{theorem}{Theorem}[section]
\newtheorem{lemma}[theorem]{Lemma}
\newtheorem{proposition}[theorem]{Proposition}
\newtheorem{corollary}[theorem]{Corollary}

\theoremstyle{definition}

\theoremstyle{remark}
\newtheorem{remark}[theorem]{Remark}

\numberwithin{equation}{section}
\def\NN{{\bf N}}

\def\dist{{\rm dist}\,}
\def\codim{{\rm codim}\,}
\def\supp{{\rm supp}\,}
\def\conv{{\rm conv}\,}
\def\Lip{{\rm Lip}\,}
\def\Re{{\rm Re}\,}
\def\Im{{\rm Im}\,}
\def\Int{{\rm Int}\,}
\def\e{\varepsilon}
\def\usim{\smash{\mathop{\sim}\limits^u}}
\begin{document}
\title[Compressions and diagonals of operator functions]
{On regularity of compressions and diagonals of operator functions}

\author{Vladimir M\"uller}
\address{Institute of Mathematics,\\
Czech Academy of Sciences,\\
 \v Zitna Str. 25, 115 67 Prague,\\
 Czech Republic}

\email{muller@math.cas.cz}

\author{Yuri Tomilov}
\address{
Institute of Mathematics\\
Polish Academy of Sciences\\
\'Sniadeckich 8\\
00-656 Warsaw, Poland \\
and \\
Faculty of Mathematics and Computer Science\\
Nicolaus Copernicus University\\
Chopin Street 12/18\\
87-100 Toru\'n, Poland
}
\email{ytomilov@impan.pl}


\thanks{
The research was supported by GA \v CR/NCN grant 25-15444K.  The first author was supported by the Czech Academy of Sciences (RVO:67985840). The second author was supported by the NCN grant 2024/06/Y/ST1/00044.
The second author  was also partially supported by the NCN grant UMO-2023/49/B/ST1/01961 and the NAWA/NSF
grant BPN/NSF/2023/1/00001.}


\begin{abstract}
Replacing operators with
continuous operator-valued functions, we prove time-dependent versions
of well-known results on compressions and diagonals
of bounded operators.
The setting of smooth functions is also addressed.
Our results have no analogues in the literature and
rely on a new technique.
The results are especially transparent for selfadjoint operators.
\end{abstract}

\maketitle
\section{Introduction}

To introduce the objects of our study,
let $L(H)$ be the space of bounded linear operators on a separable Hilbert space $H,$ and $A \in L(H).$
If $(e_k)_{k=1}^N$ is an orthonormal basis in $H, 1 \le N\le \infty,$ then $A$ admits a matrix
representation $M_A=(\langle A e_i, e_j\rangle)_{i,j=1}^N.$
If $N<\infty$ then $M_A$ is a finite matrix, and, given $A,$ it is
a central problem  of matrix theory to
 find $M_A$ most suitable for revealing the structure of $A$
and convenient for dealing with specific tasks depending
on $M_A.$
While this topic is classical, there are still many problems in understanding $M_A$
e.g. the existence of prescribed zero patterns in $M_A$,
awaiting their complete solutions.
For an infinite-dimensional space $H$ the relations between $A$ and $M_A$
become even more involved and depend on advanced methods and techniques
stemming from various domains of analysis.
As shown in \cite{MTCrelle},
the structure of $M_A$ can be rather arbitrary if one is allowed to exclude
elements of $M_A$ corresponding to thin subsets of $\mathbb N\times \mathbb N.$
On the other hand, $M_A$ may have a very particular form,
and for example, (block) tridiagonal and pentadiagonal forms
are standard tools in the study of selfadjoint and unitary operators,
respectively. As another, striking instance, recall that
a large class of bounded selfadjoint operators on $H$  having a Hankel matrix
with respect to some basis can be  characterized using spectral terms,
as shown in \cite{Treil}.

In the study of $M_A$ the main diagonal plays a distinguished role,
as it carries a lot of information about $A.$
However, extracting of this information is often very challenging
if $A$ is far from being compact.
To facilitate such a study one may consider
the whole set of diagonals $\mathcal D(A):=\{\langle A e_n, e_n\rangle_{n=1}^{\infty}\}\subset \ell_\infty(\mathbb N).$
While the study of diagonals has received a substantial attention last years,
especially for selfadjoint and normal $A,$
the structure of $\mathcal D(A)$ and its interplay with $A$
are still far from being understood.
First, note that the study of diagonals can proceed
as either the study of $\mathcal{D}(A)$ for a class of operators $A$,
or as the study of  $\mathcal{D}(A)$ for a fixed operator $A.$
Most of the relevant research has been concentrated on the first, easier problem,
but there have been several papers addressing the second problem as well.
Since the entries constituting diagonals of $A$ belong to the numerical range $W(A)$
of $A$, the numerical range appears
to be  very natural tool for describing, at least, a part of $\mathcal D(A)$.
Moreover,  the essential numerical range $W_e(A)$ of $A,$
consisting of the limits of elements in $\mathcal D(A),$
whenever they exist, looks like an appropriate substitute or
useful addition.
Both $W(A)$ and $W_e(A)$  will be basic in our considerations.

Until Kadison's papers \cite{Kadison02a, Kadison02b}, the research on $\mathcal D(A)$
was rather sporadic. To note several landmark results preceding \cite{Kadison02a, Kadison02b},
recall that in a deep article by Stout \cite{Stout81} the set
$\mathcal D(A)$ appeared in a natural way in the study of Schur algebras.
Stout related $\mathcal D(A)$
to $W_{e}(A),$ and proved that $0 \in W_e(A)$
implies that for any $(\alpha_n)_{n=1}^{\infty}
\not \in \ell_1$ there is $(d_n)_{n=1}^{\infty} \in \mathcal D(A)$ such that $|d_n|\le |\alpha_n|$ for all $n,$
thus clarifying fine structure of $\mathcal D(A).$
The existence of zero diagonals in $\mathcal D(A)$
for arbitrary $A\in L(H)$ was studied by P. Fan in \cite{Fan84}
He proved that $A$ admits a zero diagonal
if and only if there exists an orthonormal basis $(e_k)_{k=1}^\infty \subset H$
such that  $\sum_{k=1}^{n_m} \langle A e_k, e_k \rangle \to 0$ as $m \to \infty$
for a subsequence $(n_m)_{m=1}^\infty.$
Clearly, the same statement holds with any $\lambda \in \mathbb C$ in place of $0$, thereby describing constant diagonals.
The paper \cite{Fan84} inspired a closer look at the nature of $\mathcal D(A),$
and Herrero showed in \cite{Herrero91} that if $(d_n)^\infty_{n=1}$ belongs to the interior $\Int W_{e}(A)$ of $W_{e}(A)$
and $(d_n)^\infty_{n=1}$ has a limit point again in $\Int W_{e}(A)$, then $(d_n)^\infty_{n=1} \in \mathcal D(A)$.
However, there is, in general, a certain arbitrariness if one restricts to the study of diagonals for class
of operators. For instance, as shown in \cite{Fong_d}, any bounded sequence can be realized as a diagonal
of some nilpotent operator (of order $2$).

In \cite{Bourin}, Bourin put the results of Herrero into a  general context of operator-valued
diagonals. It was  proved in \cite{Bourin} that if $W_{e}(A)$  contains the open unit disc, then for every
$(D_n)_{n=1}^\infty \subset L(H)$
satisfying
$\sup_{n\ge1}\Vert D_n\Vert<1$,
there exist mutually orthogonal subspaces $M_n \subset H$ such that
$\bigoplus_{n=1}^{\infty} M_n = H$ and, denoting by $P_{M_n}$ the orthogonal projection onto $M_n,$
the operator $P_{M_n} A \upharpoonright_{M_n}$ is unitarily equivalent to $D_n$ for all $n \in \mathbb  N.$
 The direct sums $\bigoplus_{n \ge 1} P_{M_n} A \upharpoonright_{M_n},$ called ``pinchings'' of $A$ in \cite{Bourin}, can be considered as  operator counterparts of elements from $\mathcal D(A),$ and they essentially coincide with those  elements when $D_n$ act on one-dimensional spaces. A complementary discussion and related results can be found in \cite{Bourin_JFA}.
A close, but weaker statement was obtained much earlier in \cite[Theorem 4.8]{Bercovici}.
See also \cite[p. 440]{Anderson} or \cite[Lemma 2.5]{Davidson}.

A different perspective  has been opened since
Kadison's striking results on diagonals of selfadjoint projections.
Given tuples of reals $d=\{d_{i}\}_{i=1}^{n}$ and $\lambda=\{\lambda_{i}\}_{i=1}^{n}$
in non-increasing order, Schur and Horn proved that there is a Hermitian matrix with diagonal values $\{d_{i}\}_{i=1}^{n}$ and eigenvalues $\{\lambda_{i}\}_{i=1}^{n}$ if and only if $\lambda$ majorizes $d$, in the sense that
$\sum_{i=1}^{k}d_{i}\leq\sum _{i=1}^{k}\lambda_{i}$ for $k=1,\dots ,n-1$, and $\sum_{i=1}^{n}d_{i}=\sum _{i=1}^{n}\lambda_{i}.$
Looking for similar results in
infinite dimensions,
Kadison discovered in \cite{Kadison02a, Kadison02b}
that a sequence $(d_n)_{n=1}^\infty$ is
a diagonal of some \emph{selfadjoint} projection if and only if it takes values in $[0, 1]$ and
the sums $a := \sum_{d_n < 1/2} d_n$ and $b:=\sum_{d_n \ge 1/2} (1-d_n)$ satisfy either $a+b=\infty$ or $a+b <\infty$ and $a-b \in \mathbb Z.$
This is in sharp contrast with a result by A. Neumann \cite{Neumann} describing
$l^\infty$-closure of diagonals of a diagonalizable selfadjoint operator as the closure
of all permutations of the sequence of its eigenvalues. If the kernel and range of
a projection $P$ are infinite-dimensional, then one only gets $\overline{\mathcal D(P)}^{\|\cdot\|_{\infty}}=[0,1]^\NN.$
The Kadison integer condition was recognized as the Fredholm index obstruction in the subsequent paper by Arveson \cite{Arveson_USA},
where Kadison's dichotomy was studied for normal operators with finite spectrum.
Later on, the nature of the integer was identified and clarified by Kaftal and Loreaux
in  \cite{Kaftal17}
 and in the subsequent paper
\cite{Loreaux19},
as a so-called essential codimension of a pair of projections.

Developing their ideas in various directions, Arveson and Kadison
recast in \cite{Arveson06} generalizations of the Schur-Horn theorem
in the setting of operator algebras, where matrix algebras and diagonals are replaced, respectively, by finite factors and maximal Abelian self-adjoint subalgebras (MASAs),
and the Schur-Horn majorisation
is generalised properly.
They gave also rise to an intensive activity around
the structure of $\mathcal D(A)$ for classes of
$A \in L(H).$
Without mentioning all of the substantial contributions,
we give only a few samples related to  this discussion.
Pairs of null real sequences realized as sequences of eigenvalues and diagonals of positive compact operators with a trivial kernel were characterized in \cite{Kaftal10}.
A description of diagonals for selfadjoint operators $A$ with finite spectrum was given in \cite{Bownik15}.
For a class of unitary operators $A,$ the set  $\mathcal D(A)$   was characterized in \cite{Jasper}.
Very recently, an (almost) complete description of diagonals of arbitrary bounded
selfadjoint operators was obtained in the important papers \cite{Bownik_m} and \cite{Bownik_n}.

The main bottleneck in the study of  diagonals described above
was a lack of techniques taking into account presence of continuous spectrum.
Thus, even the description of elements of $\mathcal D(\mathcal M)$ for the toy multiplication operator
$(Mf)(t)=t f(t),$ $ f \in L^2([0,1]),$ was out of reach.
To remedy this problem, a new approach to the study the diagonals of bounded operators $T$
(and their tuples
$\mathcal A=\langle A_1, \dots, A_k \rangle \in B(H^k)$)
was proposed in \cite{MT19}.
Arguing in terms of geometry of numerical ranges,
\cite{MT19} introduced
a
\emph{Blaschke-type} condition
$
\sum_{n=1}^{\infty} \dist (d_n, \mathbb C \setminus W_{e} (A))=\infty
$
on the size of
$(d_n)_{n=1}^\infty \subset {\rm Int}\, W_e(A)$
near the boundary of
$W_{e}(A)$
and offered, in particular, the following two results
given in Theorem \ref{blash_th} below.
\newpage
\begin{theorem}\label{blash_th}
Let $A \in L(H).$
\begin{itemize}
\item [(i)]
 If $(d_n)_{n=1}^\infty
\subset {\rm Int} \, W_e(A)$  satisfies
\begin{equation}\label{bla}
\sum_{n=1}^\infty \dist (d_n, \mathbb C \setminus W_e(A))=\infty,
\end{equation}
then $(d_n)_{n=1}^{\infty} \subset \mathcal D(A).$
\item [(ii)] Let $W_e(A)$ contain the open unit disc.
Let $H_n, n\in\mathbb N,$ be separable Hilbert spaces (finite or infinite-dimensional)
and $D_n\in L(H_n)$ be strict contractions, i.e., $\|D_n\|<1$
for all $n \in \mathbb N.$ If
\begin{equation}\label{bla_o}
\sum_{n=1}^\infty (1-\|D_n\|)=\infty,
\end{equation}
then there is
a sequence $(M_n)_{n=1}^{\infty}$ of
 mutually orthogonal subspaces of $H$ such that
$\bigoplus_{n=1}^{\infty} M_n = H$ and $P_{M_n} A \upharpoonright_{M_n}$
is unitarily equivalent to $D_n$ for every $n \in \mathbb  N.$
\end{itemize}
\end{theorem}

(The statement in (ii) was formulated in \cite{MT19} somewhat imprecisely.)
Theorem \ref{blash_th},(i) substantially generalized Herrero's theorem,
and the approach in \cite{MT19} led to a general and new method for constructing
a big part of diagonals of operators $A \in L(H)$ that works in a variety of different and general settings,
for instance in \cite{Bownik_m} and \cite{Bownik_n} mentioned above.
For its pertinent discussion and implications, see e.g. \cite{LW20}.
The relevance of \eqref{bla} is illustrated by the fact that for the unilateral shift $S$ on $H^2(\mathbb D)$ and for the multiplication operator $M$
as above the sets $\mathcal D(S)$ and $\mathcal D(M)$  are in fact characterized by \eqref{bla},
see \cite{Herrero91} and \cite{LW20}.
Moreover, it allowed one to  construct diagonals  for fixed operators rather than operator classes.
It is crucial that the same type of technique yields the results in the context of operator diagonals
thus replacing the  uniform contractivity condition
from \cite{Bourin} on operator diagonals
by a more general assumption \eqref{bla_o}
of Blaschke's type.
Finally, the approach from \cite{MT19}  was developed in
\cite{MT_JFA1}
and \cite{MTCrelle} for describing the properties of the whole matrices
$M_A$ for $A \in L(H),$ although the main diagonals played a distinguished role there.

The present paper takes the next step ahead in
understanding the properties of matrices of $A \in L(H)$
and  related concepts.
It is a frequent and very fruitful idea in analysis to clarify whether a certain important property is
stable under perturbations in a natural sense. One of the ways to realise this is to obtain its time-dependent version. In this spirit one may mention, for instance,
the study of the existence of regular curves of eigenvalues for families $(A(t))_{t \in [0,1]} \subset L(H),$
see e.g. \cite{Baumgartel} or \cite[Chapter VII]{Kato},
or a recent work \cite{Motakis} providing a time-dependent, ``continuous'' version of an important estimate
by Bourgain and Tzafriri.
Although time-dependent versions of results on diagonals and matrix representations look
natural as such, they are essentially missing in the literature.
The only paper  in this direction, \cite{Bownik_Szysz}, provides a time-dependent
version of Kadison's theorem.
It shows that if $d_n: X \to [0,1],$ $ n \in \mathbb N,$  are measurable functions on a measure space $X,$
then there exists a measurable projection-valued function $P:X \to L(H)$ with
${\rm diag}\, (P(x))= (d_n(x))_{n=1}^{\infty}$ for almost all $x \in X$ if and only if for almost all $x\in X$
either $a(x)=\infty$ or $b(x)=\infty,$ or  $a(x), b(x)<\infty$ and $a(x)-b(x)\in \mathbb Z,$
where $a$ and $b$ are defined similarly to the above. The statement yielded a description of spectral functions
of shift-invariant subspaces in $L^2(\mathbb R^d)$, and will surely spark an interest in the near future.
 However,
the arguments in \cite{Bownik_Szysz} are tailored  for the concrete setting of Kadison's theorem,
and cannot be used in the general case of selfadjoint or, more generally, arbitrary bounded operators.
Moreover,
measurability is a rather
weak property and if one aims at more regular functions $d_n$, then a different technique is required.
This and other similar problems will be addressed in this paper.

First, we obtain the next time-dependent, ``continuous'' version of a very partial but important case
of Bourin's result mentioned above. It deals, in fact, with  continuity of classical compressions,
under the assumptions ensuring that such compressions exist.
See Remark \ref{continu} below.

\begin{theorem}\label{T9_intro}
Let $A \in C([0,1], L(H))$ be such that $0\in {\rm Int}\, W_e(A(t))$ for all $t\in [0,1]$.
Let $Z$ be a separable Hilbert space, and let $D \in C([0,1], L(Z))$
satisfy
\begin{equation}\label{t9eq}
\max_{t\in [0,1]}\|D(t)\|<\min_{t\in [0,1]}\dist \bigl(0,\partial W_e(A(t))\bigr).
\end{equation}
Then there exists $S \in C([0,1], L(Z,H))$ such that $S$ is isometry-valued and
\begin{equation}\label{compress}
 S^*(t)A(t)S(t)=D(t), \qquad t\in [0,1].
\end{equation}
\end{theorem}

Moreover, for a special type of compression with finite-dimensional $Z$
we are able to control rather general moduli of continuity, see Theorem \ref{T5} below.
If the compression $D$ is scalar, then we prove
a stronger version  of Theorem \ref{T9_intro}, dealing with  $A$ and $D$
of prescribed smoothness.
Here differentiability is understood in the \emph{real sense}, so that all derivatives are
only $\mathbb R$-linear.

\begin{theorem}\label{partition_intro}
Let $A\in L(H)$ and $\Omega\subset\mathbb C$ be an open set. For $r \in \mathbb N\cup\{0, \infty\},$ let $d:\Omega\to\Int W_e(A)$ be a $C^r$-function. Then there exists a sequence
$(u_n)_{n=1}^\infty \subset C^r (\Omega, H),$ such that
\begin{itemize}
\item [(i)] $\|u_n(\omega)\|=1, \, n\in\mathbb N, \omega\in \Omega$;

\item [(ii)] If $n\ne m,$ $ n, m \in \mathbb N,$ then $u_m(\omega)\perp u_n(\omega')$ for all $\omega,\omega'\in \Omega$;

\item [(iii)] $\langle Au_n(\omega),u_n(\omega)\rangle=d(\omega), \, n\in\mathbb N, \omega\in \Omega$;

\item [(iv)] There exists a Hilbert space $X$ and an isometry-valued function $S \in C^r(\Omega, L(X, H))$
satisfying
\[
S^*(\omega) A S(\omega)=d(\omega) I, \qquad \omega \in \Omega.
\]
\end{itemize}
\end{theorem}

Second, we prove the next non-stationary version of Herrero's theorem that was a prototype
for Theorem \ref{blash_th}, (ii).
\begin{theorem}\label{diagonal_intro}
Let $A \in C([0,1], L(H))$ and $(d_n)_{n=1}^\infty\subset C([0,1])$ be such that
\[ d_n(t)\in\Int W_e(A(t)), \qquad n\in \mathbb N, \quad t\in [0,1].
\]
If
\[
\inf\{\dist(d_n(t),\partial W_e(A(t)): n \in \mathbb N, t \in [0,1])\}>0,
\]
then there exists a sequence  $(v_n)_{n=1}^\infty \in C([0,1], H)$ such that
$(v_n(t))_{n=1}^\infty$ is an orthonormal basis in $H$ for every $t\in[0,1],$
and
$$
\langle A(t)v_n(t),v_n(t)\rangle= d_n(t), \qquad n \in \mathbb N, \quad t \in [0,1],
$$

Moreover, the family $(v_n)_{n=1}^\infty$ is equicontinuous if $(d_n)_{n=1}^\infty$ is so.
\end{theorem}

The next statement is a direct corollary of Theorem \ref{diagonal_intro}, which is worth noting.

\begin{corollary}\label{cor_diag_intro}
Let $A\in C([0,1], L(H))$ and   $d \in C([0,1], \ell^\infty(\mathbb N))$ satisfy
$$
W_e(A(t))\supset\{z:|z|\le \|d(t)\|_\infty+\theta\}
$$
for some $\theta>0$. Then there exists a continuous function $v:[0,1]\to\ell^\infty(\mathbb N, H)$
such that
$(v_n(t))_{n=1}^\infty$ is an orthonormal basis in $H$ for all $t\in [0,1]$ and
$$
\left \langle A(t)v_n(t),v_n(t)\right \rangle=d_n(t)
$$
for all $n \in \mathbb N$ and $t \in [0,1]$.
\end{corollary}

Our approach uses  several ideas from \cite{MT19}, \cite{MT_LMS} and \cite{MT_JFA}.
However, methods from these papers cannot be applied directly,
and new tools are required.
The proof of Theorem \ref{T9_intro} is based on
factorization of \eqref{compress}
into simpler germs,
where the regularity properties are easier to control,
and on inductive arguments adapted to the time-dependent
case. In the case of scalar and smooth $D$,
we employ a different argument depending on partitions
on unity.
To show Theorem \ref{diagonal_intro}, we combine the idea from \cite{MT19}
with a reasoning similar in spirit  to the proof of Theorem \ref{T9_intro},
and rely on inductive arguments again.

We give versions of our results for the important case when operators $A(t)$ are selfadjoint.
The case of selfadjoint $A$ received a special attention in the study of compressions and diagonals
due to its applications to frames theory, theory of von Neumann algebras and
majorization, and we find it useful to address it as well.
Since then ${\rm Int} \, W_e (A(t))=\emptyset$ for all $t \in [0,1]$ we have to adjust
our proofs and the corresponding formulations accordingly.
See Theorems \ref{t2self} and \ref{diagonalself} for counterparts of Theorems \ref{T9_intro}
and \ref{diagonal_intro} in this setting.

At the moment, it is not clear whether
Theorems \ref{T9_intro}, \ref{partition_intro} and \ref{diagonal_intro} can be shown
under Blaschke type assumptions of Theorem \ref{blash_th}.
We even do not know whether one can get a time-dependent variant of Bourin's formulation,
and we suspect that Blaschke type assumptions would probably require
additional new ideas.

We formulated Theorems~\ref{T9_intro} and~\ref{diagonal_intro}, as well as the related statements below, for continuous functions on the interval $[0,1]$. It is plausible that similar results hold for continuous functions on an arbitrary compact metric space. However, we restrict our attention to the unit interval in order to keep the complexity of the proofs and the length of the paper within reasonable limits.

\section{Notation and conventions}

Let $H$ be an infinite-dimensional complex separable Hilbert space, and
let $L(H)$ stand for the algebra of bounded linear operators on $H.$
The space of bounded linear operators between Hilbert spaces $K$ and $H$
is denoted by $L(K, H).$

For bounded unitarily equivalent linear operators $A$ and $B$ acting  on the corresponding Hilbert spaces we write
$A \usim B.$

For a metric space $G$ and a Banach space $X$ we denote by $C(G, X)$
the space of continuous functions on $G$ with values in $X.$ For $f\in C(G,X)$ we denote $Z(f)=\bigvee_{g\in G} f(g)$, where $\bigvee$ stands for the closed linear span.

Similarly, $C^r (\Omega, X)$ denotes the space of $X$-valued $r$-times continuously differentiable functions on an open subset $\Omega\subset\mathbb C.$

For $r>0,$ we let $\mathbb D_r:=\{z\in \mathbb C:|z|< r\},$ and write
$\mathbb D$ instead of $\mathbb D_1$ to follow the established convention.

The restriction of a linear operator $A \in L(H)$ to a subspace $M$ of $H$ is denoted by
$A{\upharpoonright}_M,$ and $P_M$ stands for the orthogonal projection from $H$ onto $M.$

For a subset $S$ of a Banach space, we denote by ${\rm conv}\, S$ its convex hull, by ${\rm Int} \, S$ its interior,
by $\overline{S}$ its closure and
 by $\partial S$ its topological boundary.

We deal only with \emph{separable complex} Hilbert spaces and we assume
that a subspace of a Hilbert space is always closed.

\section{Preliminaries}\label{prelim}

Our arguments will rely on the notions of essential numerical range and essential spectrum
of $A \in L(H).$
To put them into a proper context, recall that the numerical range of $A$ is defined
as
\[
W(A):=\left \{\langle Ax, x \rangle: x \in H,  \|x\|=1 \right \}.
\]
It is well-known that $W(A)$ is a convex subset of $\mathbb C$ and $\sigma(A) \subset \overline{W(A)}$.
At the same time, $W(A)$ may be neither closed, nor open, and its geometry may be highly complicated.
While $W(A)$ contains a lot of information on $A$ and would be a natural candidate to use
for the study of continuity properties of matrix representations of $A,$  it is highly unstable, even under rank-one perturbations.
This makes it difficult to use $W(A)$ and motivates one to invoke
counterparts of $W(A)$ more suitable for various inductive constructions.

Thus, we will use essential numerical range  of $A,$ having no such a drawback.
The set $W_e(A)$ can be considered as an approximate version of $W(A),$
and it can be defined in several equivalent ways.
Following one of them,  we define  the essential numerical range   $W_{e}(A)$  as the set of
all $\lambda \in\mathbb C$ such that there exists an orthonormal sequence $(u_n)_{n=1}^{\infty}\subset H$ with
$$
\lim_{n\to\infty}\langle  A u_n,
  u_n\rangle=\lambda.
$$
As shown in \cite{Stout81}, orthonormal sequences in this definition can be replaced by bases in $H.$
Recall that $W_{e}(A)$ is a compact and convex subset of $\overline{W(A)}$
and it is non-empty if $H$ is infinite-dimensional.
Moreover, $W_e(A+K)=W_e(A)$ for every compact operator $K.$
It is easy to show that for any $A,B \in L(H)$ one has
\[
\rho(W_e(B), W_e(A)) \le \|B-A\|,
\]
where $\rho$ stands for the Hausdorff metric on compact subsets of $\mathbb C.$
Thus the mapping
$ A \mapsto W_e(A)$ is continuous. Note that $W_e(A)$ is related to the spectral properties of $A,$
and, in particular, if $\sigma_e(A)$ stands for the essential spectrum
of $A$, then ${\rm conv} \sigma_e(A) \subset W_e(A).$

The next lemma will be  indispensable in our arguments and  be applied many times
in the sequel.
\begin{lemma}\label{codim}
Let $A \in L(H).$
\begin{itemize}
\item [(a)]  One has $\lambda \in W_{e}(A)$ if and only if
for every $\e>0$ and every  subspace $M\subset H$ of finite codimension
there is a unit vector $x\in M$ such that
\[
|\langle A x,x \rangle-\lambda|<\e.
\]
Thus,
\[
W_e(P_MA{\upharpoonright}_M)=W_e(A).
\]
\item [(b)]
If $\lambda \in\Int \, W_e(A),$
then for every subspace $M\subset H$ of finite codimension there is $x\in M$ such that $\|x\|=1$ and
$$
 \langle A x, x\rangle=\lambda.
$$
Moreover, there exists an infinite-rank orthogonal projection $P$
such that
\[
PAP=\lambda P.
\]
If $A=A^*,$ then the same properties hold for $\lambda \in \Int_{\mathbb R} W_e(A),$
where $\Int_{\mathbb R} W_e(A)$ stands for the interior of $W_e(A)$ relative to $\mathbb R.$
\end{itemize}
\end{lemma}

The proof of the first property is easy and can be found e.g. in \cite[Proposition 5.5]{MT_LMS}.
The second property is more involved,
see e.g. \cite[Corollary 4.5]{MT_JFA} and \cite[Proposition 4.1]{MT19} for its proof and other related statements.
The properties (a) and (b) are very useful in inductive constructions of
sequences in $H.$ In particular,
by absorbing all of the elements constructed after a finite number of induction steps
into a finite-dimensional subspace $F$,
one may still use $W_e(A)$ when dealing with vectors from $F^\perp.$
In particular, given $\lambda \in \Int W_e(A)$ one may find a unit vector $x \in F^\perp$ with $\langle Ax, x \rangle=\lambda.$
These properties will play a role  in this paper as well.
They will be often combined with the following simple observation.
\begin{lemma}\label{comp}
Let $K \subset H$ be a compact set.
Then for every subspace $M \subset H$ with ${\rm codim}\, M <\infty$ and every $\epsilon >0$ there exists
a subspace $L \subset M, {\rm codim} \, L<\infty$ such that for all $f \in L, \|f\|=1,$
and $x \in K$ one has
$|\langle x, f \rangle| \le \epsilon.$
\end{lemma}

\begin{proof}
Let $\{x_k: 1 \le k \le n\}$ be a finite $\e$-net for $K$.
If $L=M\cap \left(\bigvee_{k=1}^{n} x_k \right)^\perp,$
then $L \subset M, {\rm codim}\, L <\infty,$ and for every $f \in L$ with $\|f\|=1,$
\[
\sup_{x \in K}|\langle x, f \rangle|= \sup_{x \in K}\min_{1 \le k \le n} |\langle x-x_k, f \rangle| \le \epsilon.
\]
\end{proof}

The next result, crucial for our studies, is a particular case of \cite[Theorem 2.1]{Bourin}
mentioned in the introduction.
It will be of substantial use, and is thus formulated
as a separate statement.

\begin{theorem}\label{bourin}
Let $A \in L(H)$ be such that $W_e(A) \supset \mathbb D_1.$
Assume that $X$ is a separable Hilbert space, and $D\in L(X)$ satisfies $\|D\|<1$.
Then there exists a subspace $M\subset H$ such that the compression $P_MA{\upharpoonright}_M$ is unitarily equivalent to $D$.
\end{theorem}

Finally, we will introduce the  notation  pertaining to regularity properties of operator-valued functions
and relevant for our approach.
Given Hilbert spaces $X$ and $H,$ for $V\in C([0,1], L(X,H)),$ write
$$
\Lip(V):=\sup\left\{\frac{\|V(t')-V(t)\|}{|t'-t|}: \,  0\le t,t'\le 1, t'\ne t\right\}.
$$

Recall that a modulus of continuity is a non-decreasing function $\delta:[0,\infty)\to[0,\infty)$ such that $\delta(0)=\lim_{s\to 0_+}\delta(s)=0$.

 If $M$ and $M'$ are metric spaces, then a function $f:M\to M'$ is said to admit $\delta$ as a modulus of continuity if $\dist(f(m),f(m'))\le \delta(\dist(m,m'))$ for all $m,m'\in M$.
A set of functions $f_n:M\to M', \ n \in \mathbb N,$  is equicontinuous if there exists a modulus of continuity $\delta$ such that all $f_n$ admit it as a modulus of continuity.

If $M$ is compact and  $f$ is continuous, then $f$ admits the modulus of continuity $\delta_f$ defined by
$$
\delta_f(\e)=\max\{\dist(f(m),f(m')): \dist(m,m')\le\e\}.
$$
Note that if $f_1,\dots,f_k:M\to M'$ are continuous functions then $\delta:=\max\{\delta_{f_1},\dots,\delta_{f_k}\}$ is a modulus of continuity for each of $f_1,\dots,f_k$.

\section{Compressions of operator-valued functions}\label{compr}

In the following we let $H$ be an infinite-dimensional separable Hilbert space.
In this section we will prove Theorem \ref{T9_intro} as well as related statements of independent interest.
The proof of Theorem \ref{T9_intro} relies on the factorization of \eqref{compress} into
three similar relations separating time-dependent left and right hand sides of \eqref{compress}.
First, we construct a stationary, diagonal compression $D_{{\rm diag}}$ of $A(t)$  depending continuously
on $t$, see Proposition \ref{P6}. Second, we show that the bilateral shift $U$  (of infinite multiplicity) can be compressed continuously
to $D(t), t \in [0,1]$, see Proposition \ref{P8}. It remains to find $U$ among compressions of $D_{{\rm diag}}$, see Proposition \ref{P7}.
However, realization of this scheme is rather involved, and it will be divided into several incremental steps.

We start with constructing stationary diagonal compressions of $A,$ and to this aim the next simple lemma will be crucial.
\begin{lemma}\label{L1}
Let $A\in L(H)$, and assume that
$\mathbb D_r \subset W_e(A)$ for some $r>0$.
 Let $X$ be a separable Hilbert space and $D\in L(X)$ satisfy
$\|D\|<r$. If $L\subset H$ is a subspace of finite codimension, then there exists
a subspace $M\subset L$ such that the compression $P_MA{\upharpoonright}_M$ is unitarily equivalent to $D$. Equivalently, there exists an isometry $V:X\to H$ with ${\rm Im}\, V=M$ such that $V^*AV=D$.
\end{lemma}
\begin{proof}
By Lemma \ref{codim}, (a),
$$
W_e(P_LA{\upharpoonright}_L)=W_e(A)\supset \mathbb D_r.
$$
Hence $W_e(r^{-1}P_LA{\upharpoonright}_L)\supset\mathbb D_1$ and, by Theorem \ref{bourin},
there exists a subspace $M\subset L$ such that $r^{-1}P_M A{\upharpoonright}_M\usim r^{-1}D$. So $P_MA{\upharpoonright}_M$ is unitarily equivalent to $D$.
\end{proof}

The construction of stationary compressions for $A$ will be achieved via an approximation argument involving inductive
construction. To be able to implement the induction
step we need several preparations.
First we prove an approximate version of \eqref{compress} for finite-dimensional $D$,
with the modulus of continuity $\delta_V$  controlled by
$\delta_A$ and $\delta_D$.

\begin{lemma}\label{L2}
Let $A \in C([0,1], L(H))$ satisfy
$0\in{\rm Int}\, W_e(A(t))$ for all $t\in[0,1].$
Let $X$ be a finite-dimensional Hilbert space and $D \in C([0,1], L(X))$
be such that
\[ \|D(t)\|<\dist(0,\partial W_e(A(t)))\]
 for all $t\in [0,1]$.
Then for every $\e>0$ and every subspace $L \subset H$ of finite codimension there exists
 $V \in C([0,1], L(X,H))$ such that:
\begin{itemize}
\item[(i)] $V(t)$ is an isometry for all $t\in [0,1]$;
\item[(ii)] $\Lip(V)<\infty,$ and $\Lip(V)$ depends only on $\e$ and the moduli of continuity $\delta_A$ and $\delta_D$.
\item[(iii)]  $Z(V) \subset L$ and $\dim Z(V)<\infty$;
\item[(iv)] $\bigl\| V^*(t)A(t)V(t)-D(t)\bigr\|\le\e$ for all $t\in[0,1]$;

\end{itemize}
\end{lemma}

\begin{proof}
Let $k$ be the smallest positive integer such that
$$
\|A(t')-A(t)\|\le\e/4
$$
and
$$
\|D(t')-D(t)\|\le\e/2
$$
whenever $t,t'\in[0,1]$, $|t-t'|\le k^{-1}$.

Set
\[
t_j= j/k, \qquad 0 \le j \le k.
\]

By Lemma \ref{L1}, there exists an isometry $V(t_0):X\to H$ satisfying $V^*(t_0)A(t_0)V(t_0)=D(t_0)$
and ${\rm Im}\, V(t_0) \subset L$.

We construct isometries $V(t_j), j=1,\dots,k,$ such that $V^*(t_j)A(t_j)V(t_j)=D(t_j)$
by induction on $j$.

Let $j\ge 1$ and suppose that the isometries $V(t_0),\dots,V(t_{j-1}):X\to H$ have already been constructed. Let
$$
M_j=\bigl(V(t_{j-1})X\cup A(t_{j-1})V(t_{j-1})X\cup A^*(t_j)V(t_{j-1})X\bigr)^\perp,
$$
and note that $\codim M_j<\infty$. By Lemma \ref{L1}, there exists an isometry $V(t_j):X\to H$ such that
$V^*(t_j)A(t_j)V(t_j)=D(t_j)$ and ${\rm Im}\, V(t_j) \subset L\cap M_j.$ The constructed isometries $V(t_0),\dots, V(t_k)$ satisfy
$$
V(t_{j+1})X \perp V(t_{j})X, A(t_j)V(t_j)X , A^*(t_{j+1})V(t_j)X
$$
for all $j=0,1,\dots, k-1$.

For $j=0,1,\dots,k-1,$ define a function $f_j:[t_j,t_{j+1}]\to[0,\pi/2]$ by
$$
f_j(t):=\frac{\pi}{2}\cdot\frac{t-t_j}{t_{j+1}-t_j}=\frac{k\pi(t-t_j)}{2}, \qquad t \in [t_{j}, t_{j+1}].
$$
Then, for $t_j\le t\le t_{j+1},$ set
$$
V(t)=V(t_j)\cos (f_j(t)) +V(t_{j+1})\sin (f_j(t)), \qquad t \in [t_{j}, t_{j+1}].
$$
Clearly, ${\rm Im}\, V(t) \subset L.$ Note that for every $x\in X,$ we have $V(t_j)x\perp V(t_{j+1})x$, and so
$$
\|V(t)x\|^2=\cos^2 (f_j(t))\|V(t_j)x\|^2+\sin^2 (f_j(t)) \|V(t_{j+1})x\|^2=\|x\|^2, \qquad t \in [0,1].
$$
Hence $V(t)$ is an isometry for each $t\in [0,1],$ and (i) holds.

By construction, the function $V:[0,1]\to L(X,H)$ is norm-continuous.
Moreover, in each interval $(t_j,t_{j+1}),$ $V$ is differentiable and
$$
\frac{{\rm d} V}{{\rm d} t}=
-V(t_j)\sin (f_j(t))\frac{k\pi}{2}+V(t_{j+1})\cos (f_j(t))\frac{k\pi}{2}
$$
with
$$
\Bigl\|\frac{{\rm d} V}{{\rm d} \, t}\Bigr\|\le
\frac{k\pi (|\sin (f_j(t))|+|\cos (f_j(t))|)}{2}\le
k\pi.
$$
Hence $\|V(t')-V(t)\|\le k\pi |t'-t|$ for all $t,t'\in [0,1],$ and  $\Lip(V)\le k\pi.$
In particular,  $\Lip(V)$ is finite and depends only on $\e$ and moduli of continuity $\delta_A$ and $\delta_D,$
implying (ii).

Clearly
 \begin{align*}
\dim Z(V)=&\dim\bigvee\{V(t)X:0\le t\le 1\}\\
=&
\dim\bigvee\{V(t_j)X:j=0,1,\dots,k\}<\infty,
\end{align*}
and (iii) follows.

It remains to prove (iv).
For all $t_j\le s\le t_{j+1}$ and $x,x'\in X$ such that $\|x\|=\|x'\|=1$
 we have
\begin{align*}
&\left |\left \langle V^*(s)A(s)V(s)x-D(s)x, x'\right\rangle \right | \\
=&
\left|\left\langle A(s)V(t_j)x\cdot\cos (f_j(s))+A(s)V(t_{j+1})x\cdot\sin (f_j(s)),
V(s)x'\right \rangle
-\left \langle D(s)x,x'\right \rangle \right|\\
\le&
\Bigl|\left\langle \cos (f_j(s)) A(t_j)V(t_j)x,V(s)x'\right\rangle +\left\langle \sin (f_j(s)) A(t_{j+1})V(t_{j+1})x,
V(s)x'\right\rangle\\
-&\langle D(s)x,x'\rangle\Bigr|+\|A(s)-A(t_j)\|\cdot|\cos (f_j(s))|+
\|A(s)-A(t_{j+1})\|\cdot|\sin (f_j(s))|\\
\le&
\Bigl|\cos^2 (f_j(s))\Bigl\langle A(t_j)V(t_j)x, V(t_j)x'\Bigr\rangle+
\sin^2 f_j(s)\Bigl\langle A(t_{j+1})V(t_{j+1})x, V(t_{j+1})x'\Bigr\rangle\\
-&\langle D(s)x,x'\rangle\Bigr|
+\e/2\\
\le&
\left|\cos^2 (f_j(s))\bigl\langle (D(t_j)-D(s))x,x'\bigr\rangle+\sin^2 (f_j(s))
\bigl\langle (D(t_{j+1})-D(s))x,x'\bigr\rangle\right|+\e/2\\
\le&
\cos^2 (f_j(s))\|D(t_j)-D(s)\|+\sin^2 (f_j(s))\|D(t_{j+1})-D(s)\|+\e/2
\le\e,
\end{align*}
and (iv) follows.
\end{proof}

Now we show how to improve our approximation of $D$
by choosing an appropriate compression of $A$.
Apart from making approximation sharper, we keep control over the continuity of compression.

\begin{proposition}\label{P4}
Let $A \in C([0,1], L(H)).$
For $\theta>0$ and a finite-dimensional Hilbert space $X,$
let
$D\in C([0,1], L(X))$ be
such that
$$
W_e(A(t))\supset\{z: |z|\le \|D(t)\|+\theta\}, \qquad t \in [0,1].
$$
Given $\e\in (0,\theta),$ a subspace $L \subset H$ with ${\rm codim} \, L<\infty,$ and $l>0,$ assume that
an isometry-valued $V \in C([0,1], L(X,H))$ is such that
\[
Z(V) \subset L, \quad \dim Z(V)<\infty,
\quad \Lip(V)\le l,
\]
and
$$
\left\|V^*(t)A(t)V(t)-D(t)\right\|\le \e, \qquad t \in [0,1].
$$
Then for every $\e'>0$
there exist $\tilde l>0$ (depending only on $l$, $\e'$ and the moduli of continuity $\delta_A$ and $\delta_D$)
and an isometry-valued
$\widetilde V \in C([0,1], L(X,H))$ satisfying
\[
 Z(\tilde V) \subset L,
\quad \dim Z(\widetilde V)<\infty, \quad \Lip(\tilde V)\le \tilde l,
\]
\begin{equation}\label{ct}
\left\| \widetilde V^*(t)A(t)\widetilde V(t)-D(t)\right\|\le \e', \qquad t \in [0,1],
\end{equation}
and
\begin{equation}\label{cont}
\left \|\widetilde V(t)-V(t) \right\|\le \sqrt{\frac{2\e}{\theta}}, \qquad t \in [0,1].
\end{equation}
\end{proposition}

\begin{proof}
Denote by $B_X$ the closed unit ball in $X$.
Clearly the sets
\[\left \{A(t)V(t)B_X:0\le t\le 1\right \} \qquad \text{and} \qquad
\left\{A^*(t)V(t)B_X: 0\le t\le 1 \right\}\]
 are compact.
So there exist $m\in\mathbb N$ and vectors $x_1,\dots x_m\in H$ such that
$$
\min\{\|A(t)V(t)x-x_j\|:1\le j\le m\}\le \e'/2
$$
and
$$
\min\{\|A^*(t)V(t)x-x_j\|:1\le j\le m\}\le \e'/2
$$
for all $t\in [0,1]$ and $x\in B_X$. Let $a={\e}{\theta}^{-1},$
and
$$
\widetilde D(t)=\frac{D(t)+(a-1)V^*(t)A(t)V(t)}{a}, \qquad t \in [0,1].
$$
Then for every $t \in [0,1],$
\begin{align*}
\|\widetilde D(t)\|\le&\|D(t)\|+a^{-1}(1-a)\bigl\|D(t)-V^*(t)A(t)V(t)\bigr\|\\
\le& \|D(t)\|+\frac{(1-a)\e}{a}< \|D(t)\|+\frac{\e}{a}\\
=&\|D(t)\|+\theta.
\end{align*}
Let
\[
M=Z(V)^\perp\cap\{x_1,\dots,x_m\}^\perp.
\]
By Lemma \ref{L2}, there exists a norm-continuous isometry-valued function $R:[0,1]\to L(X,H)$ such that
$Z(R)\subset M \cap L$, $\dim Z(R)<\infty$,
and
$$
\left \|R^*(t)A(t)R(t)-\widetilde D(t)\right\|\le\frac{\e'}{2}
$$
for all $t\in [0,1]$. In particular,
$$
Z(R)\perp Z(V),\qquad Z(R)\perp\{x_1,\dots,x_m\},
$$
Moreover, $\Lip(R)<\infty,$ and $\Lip(R)$ depends only on $\e'$ and the moduli of continuity $\delta_A$ and $\delta_{\widetilde D}$, which in turn depends only on $\e'$, $l$, $\delta_A$ and $\delta_D$.

Define
\begin{equation}\label{defv}
\widetilde V(t)=\sqrt{1-a}\,V(t)+\sqrt{a}\,R(t), \qquad t \in [0,1].
\end{equation}
By construction, $Z(\widetilde V) \subset L.$ Moreover, for every $t \in [0,1],$ taking into account that $R(t)x\perp V(t)x, x\in X,$ we have
\[
\|\widetilde V(t)x\|^2=(1-a)\|V(t)x\|^2+a\|R(t)x\|^2=\|x\|^2
\]
 for all  $x\in X$. So $\widetilde V(t)$ is an isometry on $X$.
In addition, $Z(\widetilde V)\subset Z(V)+Z(R),$ so $\dim Z(\widetilde V)<\infty$.

To prove \eqref{ct}, observe that for all $x,x'\in X$ such that $\|x\|=\|x'\|=1,$
\begin{align*}
&\Bigl|\left \langle \bigl(\widetilde V^*(t)A(t)\widetilde V(t)-D(t)\bigr)x,x' \right \rangle\Bigr| \\
=&
\Bigl|\Bigl\langle \sqrt{1-a}A(t)V(t)x+\sqrt{a}A(t)R(t)x,\sqrt{1-a}V(t)x'+\sqrt{a}R(t)x'\Bigr\rangle-\langle D(t)x,x'\rangle\Bigr|\\
\le&
\Bigl|(1-a)\bigl\langle A(t)V(t)x,V(t)x'\bigr\rangle+a\bigl\langle A(t)R(t)x,R(t)x'\bigr\rangle-\langle D(t)x,x'\rangle\Bigr|\\
+&\sqrt{a(1-a)}\left|\langle A(t)V(t)x,R(t)x'\rangle\right|+\sqrt{a(1-a)}\left|\langle A(t)R(t)x,V(t)x'\rangle\right|,
\end{align*}
where
$$
\left|\langle A(t)V(t)x,R(t)x'\rangle\right|\le
\min\left\{\left\|A(t)V(t)x-x_j\right\|:1\le j\le m\right\}\cdot\|R(t)x'\|\le\e'/2,
$$
and, similarly,
$$
\bigl|\langle A(t)R(t)x,V(t)x'\rangle\bigr|\le\e'/2.
$$
Furthermore,
\begin{align*}
&\bigl|(1-a)\langle A(t)V(t)x,V(t)x'\rangle+a\langle A(t)R(t)x,R(t)x'\rangle-\langle D(t)x,x'\rangle\bigr|\\
\le&
\bigl|(1-a)\langle A(t)V(t)x,V(t)x'\rangle +a\langle\widetilde D(t)x,x'\rangle-\langle D(t)x,x'\rangle\bigr|+{\e'/2}\\
=&\e'/2.
\end{align*}
Hence, since $\sqrt{a(1-a)}\le 1/2$,
$$
\|\widetilde V^*(t)A(t)\widetilde V(t)-D(t)\|\le\e', \qquad t \in [0,1],
$$
i.e., \eqref{ct} holds.

To show \eqref{cont}, we note that for all $t\in [0,1]$ and $x\in H$, $\|x\|=1,$
\begin{align*}
\|\widetilde V(t)x-V(t)x\|^2=&
(1-\sqrt{1-a})^2\|V(t)x\|^2+{a}\|R(t)x\|^2\\
=&
(1-\sqrt{1-a})^2+{a}= 2-2\sqrt{1-a}\le 2{a}.
\end{align*}
Therefore,
\[
\|\tilde V(t)-V(t)\|\le\sqrt{2a}=\sqrt{\frac{2\e}{\theta}}, \qquad t \in [0,1].
\]

Finally, if $t,t'\in[0,1],$ then
\begin{align*}
\|\widetilde V(t')-\widetilde V(t)\|\le&
\sqrt{1-a}\|V(t')-V(t)\|+\sqrt{a}\|R(t')-R(t)\|\\
\le& l|t'-t|+\sqrt{a} \cdot \Lip(R) |t'-t|\\
\le& \widetilde l|t'-t|
\end{align*}
for some constant $\widetilde l$ depending only on $l$, $\e'$ and the moduli of continuity $\delta_A$ and $\delta_D$.
\end{proof}

Before implementing the factorization argument in Proposition \ref{P6} below,
we note a related statement with $D$ acting on a finite-dimensional space.
In this case, we are able to provide better control
for regularity of $V,$ which will be crucial in
construction of diagonals for $A$ in the next section.

\begin{theorem}\label{T5}
Let
$X$ be a Hilbert space with $\dim X<\infty$,
$A \in C([0,1], L(H))$, $D \in C([0,1], L(X))$
and
$$
W_e(A(t))\supset \{z:|z|\le \|D(t)\|+\theta\}
$$
for all $t\in[0,1]$. Then for every subspace $L$ of $H$ with ${\rm codim}\, L<\infty,$ there exists
an isometry-valued function $V \in C([0,1], L(X,H)),$ $Z(V)\subset L,$ with modulus of continuity $\delta,$
depending only on $\theta$ and the moduli of continuity
$\delta_A$ and $\delta_D$,
such that
$$
V^*(t)A(t)V(t)=D(t), \qquad t \in [0,1].
$$
\end{theorem}

\begin{proof}
By Lemma \ref{L2}, there exists $l_1>0$ (depending only on $\theta$ and the moduli of continuity $\delta_A$ and $\delta_D$) and an isometry-valued function  $V_1 \in C([0,1], L(X,H))$ such that $Z(V_1)\subset L$, $\dim Z(V_1)
<\infty $, $\Lip(V_1)\le l_1$ and
\[
\bigl\|V^*_1(t)A(t)V_1(t)-D(t)\bigr\|\le \theta/2
\]
for all $t\in[0,1]$.
By
an inductive argument, choosing $\varepsilon=2^{-n} \theta$
and $\varepsilon'=2^{-n-1}\theta$ in Proposition \ref{P4} for every $n \in \mathbb N,$ we
find $l_{n}>0$ (depending only on $\theta$, $\delta_A$ and $\delta_D$) and an  isometry-valued function $V_n \in C([0,1], L(X,H)), n \ge 2,$ such that
\begin{equation}\label{z}
Z(V_n)\subset L, \quad \dim Z(V_n)<\infty, \quad \Lip(V_n)\le l_n,
\end{equation}
\begin{equation}\label{d}
\bigl\|V_n^*(t)A(t)V_n(t)-D(t)\bigr\|\le\theta/2^n,
\end{equation}
and
\begin{equation}\label{v}
\|V_{n+1}(t)-V_n(t)\|\le 2^{\frac{1-n}{2}}
\end{equation}
for all $n \ge 2$ and $t\in[0,1]$.

From \eqref{v} it follows that the series $\sum_{n=1}^\infty \left \| V_{n+1}(t)-V_n(t) \right \|$
converges uniformly on $[0,1].$ Hence setting 
\[
V(t)=\lim_{n\to\infty} V_n(t), \qquad t \in [0,1],
\]
we infer that $V \in C([0,1], L(X, H)).$ Since $V_n$ is isometry-valued and $Z(V_n)\subset L$ for every $n \in  \mathbb N,$  the same holds 
for $V.$ Moreover, in view of \eqref{d}, we have
$$
V^*(t)A(t)V(t)
=\lim_{n\to\infty} V_n^*(t) A(t)V_n(t)=D(t)
$$
for all $t\in [0,1]$.

Let $\e>0$ and $n_0\in\mathbb N$. If $t,t'\in[0,1]$, $|t-t'|\le \e,$ then
\begin{align*}
\|V(t)-V(t')\|\le&
\|V(t)-V_{n_0}(t)\|+\|V_{n_0}(t)-V_{n_0}(t')\|\\
+&\|V_{n_0}(t')-V(t')\|
\le
2\cdot\sum_{n=n_0}^\infty 2^{\frac{1-n}{2}}+l_{n_0}\e\\
\le& \frac{2^{(5-n_0)/2}}{2-\sqrt 2}+l_{n_0}\e
\end{align*}
Set
\[
\delta(\e)=\inf_{n_0\in\mathbb N}\left(\frac{2^{(5-n_0)/2}}{2-\sqrt 2}+l_{n_0}\e \right).
\]
Then $V$ admits $\delta$ as a modulus of continuity.

\end{proof}

Now we have all the required tools to realize the first step of
our factorization by finding a diagonal operator among
continuous compressions of $A.$

\begin{proposition}\label{P6}
Let $A \in C([0,1], L(H))$ and assume that $0\in{\rm Int}\, W_e(A(t))$ for all $t\in[0,1]$.
Let $X$ be an infinite-dimensional separable Hilbert space and let $D\in L(X)$ be a diagonal operator satisfying
\[
\|D\|<\min_{t\in[0,1]} \dist(0,\partial W_e(A(t))).
\]
Then there exists an isometry-valued mapping $V \in C([0,1], L(X,H))$ such that
\[
V^*(t)A(t)V(t)=D, \qquad  t\in [0,1].
\]
\end{proposition}

\begin{proof}
Choose $\theta>0$ such that
\[
W_e(A(t))\supset \{z\in\mathbb C:|z|\le \|D\|+\theta\}, \qquad t\in[0,1].
\]

Let $(x_j)_{j=1}^\infty$ be an orthonormal basis in $X$ and let $(d_j)_{j=1}^\infty\subset\mathbb C$ be such that
$Dx_j=d_jx_j$ for all $j\ge 1$.
Set
\[
X_k=\bigvee\{x_j:j\le 2^k\}
\qquad\text{and}\qquad
D_k=P_{X_k}D{\upharpoonright}_{X_k}, \qquad k \in \mathbb N.
\]

As in the proof of Theorem~\ref{T5}, we first use Lemma~\ref{L2} and then apply Proposition~\ref{P4} successively with
$\e=2^{-(n-1)}\theta$ and $\e'=2^{-n}\theta$, $n \ge 2$, in order to obtain a sequence $(l_n)_{n=1}^\infty\subset(0,\infty)$
depending only on $\theta$ and $\delta_A$.

We now construct inductively isometry-valued continuous mappings
\[
V_{k,n}:[0,1]\to L(X_k,H), \qquad n\le k,
\]
such that
\begin{itemize}
\item[(i)] $\dim Z(V_{k,n})<\infty$;
\item[(ii)] $\bigl\|V_{k,n}^*(t)A(t)V_{k,n}(t)-D_k\bigr\|\le 2^{-n} \theta (2-2^{-k})$;
\item[(iii)] $\bigl\|V_{k,n}(t)-V_{k,n-1}(t)\bigr\|\le 2^{\frac{2-n}{2}}$;
\item[(iv)] $\Lip(V_{k,n})\le l_n$;
\item[(v)] $V_{k+1,n}$ is an extension of $V_{k,n}$.
\end{itemize}

First, Lemma~\ref{L2} yields $V_{1,1}$ with properties \textnormal{(i)}–\textnormal{(v)}.

Assume now that for some $k\in\mathbb N$ we have constructed $V_{k',n}$ satisfying \textnormal{(i)}–\textnormal{(v)}
for all $n$ and $k'$ with $n\le k'\le k$.
Set
\[
Y_k=X_{k+1}\ominus X_k
\qquad\text{and}\qquad
D'_k= P_{Y_k}D{\upharpoonright}_{Y_k}.
\]

Define
\[
L_k= \bigcap_{n=1}^k Z(V_{k,n})^\perp,
\]
so that ${\rm codim}\, L_k <\infty.$
Using the compactness of the set
$$
\left\{A(t)V_{k,n}x, A^*(t)V_{k,n}(t)x: t\in[0,1], 1\le n\le k, x\in X_k, \|x\|=1\right\}
$$
and Lemma \ref{comp}, we find a subspace  $M_k\subset L_k$ with $\codim M_k<\infty$ such that
\[
\left|\langle A(t)V_{k,n}x,m\rangle\right|<\theta\cdot 2^{-2k-2}
\qquad
\text{and} \qquad
\left|\langle A^*(t)V_{k,n}x,m\rangle\right|<\theta\cdot 2^{-2k-2}
\]
for all $t\in[0,1], 1\le n\le k,$ and $x\in X_k, m\in M_k,$ with $\|x\|=\|m\|=1$.

\smallskip

\emph{Step 1: construction of $V_{k+1,1}$.}
By Lemma~\ref{L2} there exists an isometry-valued mapping $R_{k,1}:[0,1]\to L(Y_k, H)$ such that
$Z(R_{k,1})\subset M_k$,
\[
\dim Z(R_{k,1})<\infty, \qquad  \Lip(R_{k,1})\le l_1,
\]
and
\[
\|R^*_{k,1}(t)A(t)R_{k,1}(t)-D'_k\|\le \theta/2, \qquad t \in [0,1].
\]
For $t \in [0,1]$ set
\[
V_{k+1,1}(t)=V_{k,1}(t)\oplus R_{k,1}(t).
\]
Arguing as in Proposition~\ref{P4}, one checks that $V_{k+1,1}$ satisfies \textnormal{(i)}–\textnormal{(v)} (for $n=1$).

\smallskip

\emph{Step 2: construction of $V_{k+1,j}$ for $2\le j\le k$.}
Assume that $V_{k+1,1},\dots,V_{k+1,j-1}$ have already been defined for some $2\le j\le k$.
Using Proposition~\ref{P4}, we obtain a continuous isometry-valued mapping
$R_{k,j}:[0,1]\to L(Y_k,H)$ such that
\[
Z(R_{k,j})\subset M_k, \qquad \dim Z(R_{k,j})<\infty, \qquad  \Lip (R_{k,j})\le l_j,
\]
\[
\|R^*_{k,j}(t)A(t)R_{k,j}(t)-D'_k\|\le \theta/2^j
\quad\text{and}\quad
\|R_{k,j}(t)-R_{k,j-1}(t)\|\le 2^{\frac{2-j}{2}}
\]
for all $t \in [0,1]$.
We then set
\[
V_{k+1,j}(t)= V_{k,j}(t)\oplus R_{k,j}(t), \qquad t \in [0,1].
\]
Again $V_{k+1,j}$ satisfies \textnormal{(i)}–\textnormal{(v)}.

\smallskip

\emph{Step 3: construction of $V_{k+1,k+1}$.}
Once $V_{k+1,1},\dots,V_{k+1,k}$ are available, Lemma~\ref{L2} provides $V_{k+1,k+1}$ satisfying
\textnormal{(i)}–\textnormal{(v)}.
This completes the inductive step in $k$.

\smallskip

We have thus constructed continuous functions $V_{k,n}:[0,1]\to L(X_k,H)$ satisfying \textnormal{(i)}–\textnormal{(v)}
for all $k,n\in\mathbb N$ with $n\le k$.

Fix $n\in\mathbb N$ and $t\in[0,1]$. For $k\ge n$ the operator $V_{k+1,n}(t)$ extends $V_{k,n}(t)$.
Define an isometry
\[
V_n(t): \bigcup_{k=1}^\infty X_k \to H
\]
by
\[
V_n(t)x =V_{k,n}(t)x \quad \text{if } x \in X_k,
\]
and extend $V_n(t)$ by continuity to an isometry $V_n(t):X\to H$, keeping the same notation.
Moreover, for $t,t'\in[0,1]$ and $x\in\bigcup_{k=1}^\infty X_k$ with $\|x\|=1$ we have
\[
\|(V_n(t)-V_n(t'))x\|\le l_n|t-t'|,
\]
and hence the same inequality holds for all $x\in X$ with $\|x\|=1$.
Thus $V_n$ is continuous on $[0,1]$ and $\Lip(V_n)\le l_n$.

By construction we also have, for all $t\in[0,1]$,
\begin{equation}\label{lim}
\|V^*_n(t)A(t)V_{n}(t)-D\|\le 2 \cdot \theta/2^n
\qquad \text{and} \qquad
\|V_{n+1}(t)-V_n(t)\|\le 2^{\frac{1-n}{2}}.
\end{equation}

Arguing as in the proof of Theorem~\ref{T5} and using the second estimate in
\eqref{lim}, we define
\[
V(t)=\lim_{n\to\infty} V_n(t), \qquad t\in[0,1],
\]
and infer that $V\in C([0,1],L(X,H))$ and that $V$ is isometry-valued.
Finally, the first estimate in \eqref{lim} implies that
$V^*(t)A(t)V(t)=D$ for all $t\in[0,1]$.
\end{proof}

We proceed with the second step of our factorization and construct $D$
as a continuous compression of a fixed unitary operator $U,$
being a bilateral shift of infinite multiplicity.
In other words, $U$ will be a dilation of  $D(t)$ for every $t \in [0,1],$
 on a space depending continuously on $t.$
This will require to impose additional restriction
on the size $D,$ which is only of a technical character.
For the arguments employed in the next proposition,
recall that a unitary operator $U$ on a Hilbert space $Y$ is said to be a bilateral shift of multiplicity $\mu$ if there exists a subspace $Z$ of $Y$ with $\dim Z=\mu,$  a wandering subspace
of $U,$ such that
$Y=\bigoplus_{j\in\mathbb Z}U^jZ,$ where $\bigoplus$ denotes the orthogonal sum.

\begin{proposition}\label{P8}
Let $X$ and $Y$ be separable Hilbert spaces such that $\dim X=\dim Y=\infty.$
Suppose that $U\in L(Y)$ is the bilateral shift of infinite multiplicity
and $D\in C([0,1], L(X))$ satisfies $\max\{\|D(t)\|:t\in[0,1]\}<1$. Then there exists  $W \in C([0,1], L(X, Y))$ such that
$W$ is isometry-valued and
\[
W^*(t)UW(t)=D(t), \qquad t\in[0,1].
\]
\end{proposition}

\begin{proof}
Let $Z\subset Y$ be a wandering subspace for $U$, i.e., $Y=\bigoplus_{j\in\mathbb Z}U^jZ$.
Let $W_0:X\to Z$ be a surjective isometry. For every $t \in [0,1]$
consider the defect operator $\Delta_{D(t)}:=\left(I-D^*(t)D(t)\right)^{1/2},$
and define the operator $W(t): X \to Y$
by
$$
W(t)x:=\sum_{j=0}^\infty U^{-j}W_0 \Delta_{D(t)}D^j(t)x, \qquad x \in X.
$$

By the binomial series,
 $(1-z)^{1/2}=\sum_{n=0}^{\infty} \alpha_n z^n,$$ z \in \mathbb D,$
where
\[
 \alpha_n = \binom{1/2}{n} (-1)^{n}
= -\left|\frac{ (-1/2) (1-1/2) \dots (n-1-1/2)}{n!}\right|  \le 0, \qquad n \ge 1.
\]
 Thus we have $\sum_{n=0}^\infty |\alpha_n|<\infty$ and  $\Delta_{D(t)}$ can be defined by the Riesz-Dunford calculus for every $t \in [0,1].$
Note that $\Delta_{D(\cdot)} \in C([0,1], L(X))$ as the sum of the series converging uniformly on $[0,1],$ Moreover,
from the series expansion, $\Delta_{D(t)}$ is nonnegative
for every $t \in [0,1],$ since
 \[
\langle  \Delta_{D(t)}x , x \rangle \ge (1- \|D(t)\|^2)^{1/2} \ge 0,
\qquad x \in X, \|x\|=1.
 \]
In addition, estimating the series, we infer that  $\|\Delta_{D(t)}\|\le 1$.

We have
\begin{align*}
\|W(t)x\|^2=&
\sum_{j=0}^\infty\|U^{-j}W_0 \Delta_{D(t)}D^j(t)x\|^2
=\sum_{j=0}^\infty \|\Delta_{D(t)}D^j(t)x\|^2\\
=&
\sum_{j=0}^\infty (\|D^{j}(t)x\|^2-\|D^{j+1}(t)x\|^2)=
\|x\|^2-\lim_{j\to\infty}\|D^{j+1}(t)x\|^2=\|x\|^2.
\end{align*}
So $W(t)$ is an isometry for each $t\in [0,1]$.

Let $c=\max_{t\in [0,1]}\|D(t)\|,$  so that $c<1$ by assumption.
Given $\eta>0,$ there exists $\delta>0$ such that $\|D(t')-D(t)\|<\eta$ and $\|\Delta_{D(t')}- \Delta_{D(t)}\|<\eta$ for all $t,t'\in[0,1]$ with $|t'-t|<\delta$.
Then
\begin{align*}
&\|W(t')-W(t)\|\\
\le&
\sum_{j=0}^\infty \|\Delta_{D(t')}D^j(t')-\Delta_{D(t)}D^j(t)\|\\
\le&
\sum_{j=0}^\infty \|\Delta_{D(t')}-\Delta_{D(t)}\|\cdot\|D(t')\|^j+
\sum_{j=0}^\infty \|\Delta_{D(t)}\|\cdot \|D^j(t')-D^j(t)\|\\
\le&
\eta\sum_{j=0}^\infty c^j+
\sum_{j=0}^\infty
\Bigl\|\sum_{r=0}^{j-1} D^{j-1-r}(t')(D(t')-D(t))D^r(t)\Bigr\|\\
\le&
\frac{\eta}{1-c}+
\sum_{j=0}^\infty
\sum_{r=0}^{j-1} \bigl\|D^{j-1-r}(t')(D(t')-D(t))D^r(t)\bigr\|\\
\le&
\frac{\eta}{1-c}+\eta\sum_{j=0}^\infty jc^{j-1}\\
\le&
\frac{\eta}{1-c}+\frac{\eta}{(1-c)^2}.
\end{align*}
So $W \in C([0,1], L(X, Y)).$

Moreover, for all $t\in[0,1]$ and $x,x'\in X,$
\begin{align*}
\langle W^*(t)UW(t)x,x'\rangle
=&
\left\langle U\sum_{j=0}^\infty U^{-j}W_0 \Delta_{D(t)}D^j(t)x,\sum_{j=0}^\infty U^{-j}W_0 \Delta_{D(t)}D^j(t)x'\right\rangle\\
=&\sum_{j=0}^\infty \left \langle W_0 \Delta_{D(t)}D^{j+1}(t)x, W_0 \Delta_{D(t)}D^j(t)x'\right \rangle\\
=&\sum_{j=0}^\infty \left\langle \left(I-D^*(t)D(t)\right)D^{j+1}(t)x,D^j(t)x'\right\rangle\\
=&
\sum_{j=0}^\infty \left\langle\left(D^{*j}(t)D^j(t)-D^{*(j+1)}(t)D^{j+1}(t)\right)D(t) x,x'\right \rangle\\
=&\left \langle D(t)x,x' \right \rangle,
\end{align*}
i.e., $W^*(t)UW(t)=D(t),$ as required.
\end{proof}

To carry out the third step of our factorization argument,
we show how to obtain preassigned unitary compressions
of an appropriate diagonal operator.
We will need the next useful statement given in \cite[Proposition 1.1]{PokrzywaJOT},
allowing us to identify a convex combination of a finite number of operators with a compression of their direct sum.
\begin{lemma}\label{pokrzywa}
Let $(A_j)_{j=1}^n$ and $(\tilde A_j)_{j=1}^{n}$ be elements of $L(H)^n$ such that $A_j \usim \tilde A_j, 1 \le j \le n,$ and let $\alpha_1,\dots,\alpha_n\ge 0$, $\sum_{j=1}^n\alpha_j=1$. Then there exists
a subspace $M\subset \bigoplus _{j=1}^n H$ such that
\[ P_M\Bigl(\bigoplus_{j=1}^n A_j\Bigr)\upharpoonright_M \usim  \sum_{j=1}^n\alpha_j \tilde A_j.\]
\end{lemma}
(For a generalization of Lemma \ref{pokrzywa} for operator tuples see \cite[Lemma 6.1]{MT_LMS}.)

The following statement is essentially known,
see e.g. a remark following \cite[Theorem 2.2]{Davidson} and \cite[Theorem 2.1]{Fong1}.
As the constants are crucial here, we provide a detailed and explicit proof.

\begin{proposition}\label{P7}
Let $Y$ be a separable infinite-dimensional Hilbert space, and let $U\in L(Y)$ be
a unitary operator.
For every $\e>0,$  there exist a separable Hilbert space $X,$ a diagonal operator $D_{{\rm diag}}\in L(X)$
with $\|D_{{\rm diag}}\|\le 1+\e,$
and an isometry $V:Y\to X$ such that $V^*D_{{\rm diag}}V=U$.
\end{proposition}

\begin{proof}
Let $\e>0$ be fixed.
By the Weyl-von Neumann theorem (see e.g. \cite[Theorem 39.4]{Conway_o}), we can write $U=\widetilde D+K$ where $\widetilde D$ is a diagonal operator on $Y$, $\widetilde D={\rm diag}\, [\widetilde d_j],$ where $\sup_{j\ge 1}
|\widetilde d_j|\le 1,$ and $K$ is a compact operator on $Y$ with $\|K\|\le\e/2$.
Then $K=\Re K+i\Im K,$ where $\Re K, \Im K$ are selfadjoint compact operators on $Y$. So $\Re K$ and $\Im K$ are diagonal (in some orthonormal but possibly different bases):
\[
\Re K={\rm diag}\, \left[d_{j}^{(1)} \right], \qquad
\Im K={\rm diag}\, \left[d_{j}^{(2)} \right],  \qquad \sup_{j\ge 1}\left\{|d_{j}^{(1)}|, |d_{j}^{(2)}|
\right \}\le\e/2.
\]
Let
\[
D_1=(1+\e)\widetilde D, \qquad D_2=\frac{2(1+\e)}{\e}\Re K, \qquad D_3=\frac{2i(1+\e)}{\e}\Im K,
\]
and define $X=Y\oplus Y \oplus Y$ and  $D_{{\rm diag}}=D_1\oplus D_2\oplus D_3\in L(X)$. Then $D_{{\rm diag}}$ is a diagonal operator on $X$ such that $\|D_{{\rm diag}}\|\le 1+\e,$ and $U$ is a convex combination of $D_1,D_2$ and $D_3:$
$$
U=
\frac{1}{1+\e}D_1+ \frac{\e}{2(1+\e)}D_2+\frac{\e}{2(1+\e)}D_3.
$$
By Lemma \ref{pokrzywa}, there exists an isometry $V:Y\to X$ satisfying $V^*D_{{\rm diag}}V=U$.
\end{proof}

Combining Propositions \ref{P6}, \ref{P7} and \ref{P8}
and factorizing \eqref{compress} accordingly,
we derive Theorem \ref{T9_intro}, one of the main results of this paper.
Now it becomes a direct consequence of the propositions mentioned  above.
\medskip

\emph{Proof of Theorem \ref{T9_intro}}\,\,
Consider $A \in C([0,1], L(H))$ with $0\in {\rm Int}\, W_e(A(t))$ for all $t\in [0,1]$.
We have to show that if  $D \in C([0,1], L(Z))$
and $D$ is so small that \eqref{t9eq} holds, then
there exists a norm continuous, isometry-valued function $S: [0,1]\to L(Z, H)$
satisfying
$S^*(t)A(t)S(t)=D(t)$ for all $t\in [0,1].$

Choose $b_1$ and $b_2$ such that
$$
\max_{t\in [0,1]}\|D(t)\|<b_1<b_2<\min_{t\in [0,1]}\dist(0,\partial W_e(A(t))).
$$
Let $U\in L(Y)$ be a bilateral shift of infinite multiplicity and $D_{{\rm diag}}\in L(X)$ be the diagonal operator
constructed by Proposition \ref{P7} such that
 $\|D_{{\rm diag}}\|\le b_2/b_1$ and  $J^*D_{{\rm diag}} J=U$ for an isometry $J:Y\to X$.

By Proposition \ref{P6}, there exists an isometry-valued function $V \in C([0,1], L(X,H))$ such that
$$
V^*(t)A(t)V(t)=b_1 D_{{\rm diag}}, \qquad t \in [0,1].
$$

Moreover, by Proposition \ref{P8}, there exists an isometry-valued function $W \in C([0,1], L(Z,Y))$
satisfying
$$
W^*(t)UW(t)= b_1^{-1}D(t), \qquad t \in [0,1].
$$

Let $S(t)=V(t)JW(t), t \in [0,1]$.
By construction, $S \in C([0,1], L(Z, H)),$ $S$  is isometry-valued, and
for all $t\in[0,1],$
\begin{align}
S^*(t)A(t)S(t)=&
W^*(t)J^*V^*(t)A(t)V(t)JW(t)\label{eqq} \\
=&
b_1W^*(t)J^*D_{{\rm diag}} JW(t)
=
b_1W^*(t)UW(t) \notag \\
=&
D(t). \notag
\end{align}
This completes the proof.
$\hfill \qed$

\begin{remark}\label{continu}
It might be instructive to clarify that Theorem \ref{T9_intro}
addresses compressions of $A,$ which are continuous in a natural sense.
 Under the assumptions of Theorem \ref{T9_intro},
denoting the unitary operator $S(t): Z \to {\rm Im}\, S(t)$ by $S_Z(t)$
and observing that $P(t)=S(t)S^*(t):H \to H$ is the projection onto ${\rm Im}\, S(t)$ for every $t \in [0,1],$ the relation \eqref{eqq} rewrites as
\begin{equation}
S_Z^{-1}(t) \left(P(t)A(t) \upharpoonright_{{\rm Im}\, S(t)}\right)S_Z(t)=D(t), \qquad t \in [0,1].
\end{equation}
So $P(t)A(t)\upharpoonright_{{\rm Im}\, S(t)}\usim D(t)$ for all $t\in[0,1]$. Moreover, the spaces ${\rm Im}\, S(t)$ change continuously in the following sense.
If $\mathcal G(H)$ stands for the metric space of subspaces of $H,$ and $g$ denotes  the gap metric in $\mathcal G(H),$
 then
\begin{equation}
\|P(t)-P(s)\|=g({\rm Im} \, S(t), {\rm Im}\, S(s) )\le \| S(t)-S(s)\|, \qquad t,s \in [0,1],
\end{equation}
(see e.g. \cite[Lemma II.10.12 and p. 145, C.10.5]{Muller}), so that $P, S_Z, S_Z^{-1}$ as well
as $G_S: [0,1] \to \mathcal G(H),$
$G_S(t)= {\rm Im}\, S(t),$ are continuous and have modulus of continuity dominated by the modulus of continuity of $S.$
\end{remark}

\section{Smooth compressions}

This section is devoted to the proof of Theorem \ref{partition_intro},
where we replace continuous compressions by smooth ones at the price
of assumptions much stronger than those of Theorem \ref{T9_intro}.
\medskip

\emph{Proof of Theorem \ref{partition_intro}}
Let $r \in \mathbb N$ be fixed.
Given $A\in L(H),$ an open set $\Omega\subset\mathbb C$, and $d \in C^r(\Omega)$ with $d(\Omega)\subset  \Int W_e(A),$ we construct $(u_n)_{n=1}^\infty \subset C^r (\Omega, H)$ such that
$(u_n(\omega))_{n=1}^\infty$ is an orthonormal sequence in $H$ for every $\omega \in \Omega,$
and
\[
\langle Au_n(\omega),u_n(\omega)\rangle=d(\omega), \qquad n\in\mathbb N, \omega\in \Omega.
\]

Moreover, we show
that there exists an infinite-dimensional Hilbert space $X$ and an isometry-valued function $S \in C^r(\Omega, L(X, H))$
satisfying
\[
S^*(\omega) A S(\omega)=d(\omega) I, \qquad \omega \in \Omega.
\]

For $a,a',a''\in \Int W_e(A)$ let $T(a,a',a''):=\hbox{Int}\,\conv\{a,a',a''\}$.
Clearly $\{d^{-1}(T (a,a',a'')) : a,a',a''\in \Int W_e(A)\}$
is an open cover of $\Omega$. Then there exists a $C^\infty$-partition of unity subordinate to this cover, i.e., there exists a sequence $(f_k)_{k=1}^\infty \subset C^\infty(\Omega,[0,1])$ such that
\begin{itemize}
\item[(a)] $\sum_{k=1}^\infty f_k(\omega)=1$ for all $\omega\in \Omega$;
\item[(b)] for each $k\in\mathbb N$ there exist $a_k^{(1)},a_k^{(2)},a_k^{(3)}\in \Int W_e(A)$ such that
$\supp f_k\subset
 d^{-1}\left(T\left(a_k^{(1)},a_k^{(2)},a_k^{(3)}\right)\right)$;
\item[(c)] for each $\omega\in \Omega$ there exists a neighbourhood $\Omega_\omega$ of $\omega$ such that the set $\{k: \supp f_k\cap \Omega_\omega\ne\emptyset\}$ is finite.
\end{itemize}
Using Lemma \ref{codim},(a) find inductively unit vectors $v_{n,k}^{(i)}\in H, n,k\in\mathbb N, i=1,2,3$, satisfying
\begin{equation}\label{ak}
\langle Av_{n,k}^{(i)},v_{n,k}^{(i)}\rangle = a_k^{(i)}
\end{equation}
 for all $n,k,i,$ and
\begin{equation}\label{orthog}
v_{n,k}^{(i)}\perp v_{m,l}^{(j)}, \ Av_{m,l}^{(j)}, \  A^*v_{m,l}^{(j)},
\end{equation}
whenever $(n,k,i)\ne (m,l,j)$.

If $z\in T\left(a_k^{(1)},a_k^{(2)}, a_k^{(3)}\right)$ for some $k \in \mathbb N,$ then $z$ can be written uniquely as a convex combination of $a_k^{(1)},a_k^{(2)}, a_k^{(3)}:$
\[
z=\beta_k^{(1)}(z)a_k^{(1)}+\beta_k^{(2)}(z)a_k^{(2)}+\beta_k^{(3)}(z)a_k^{(3)},
\]
 where  $\beta_k^{(1)},\beta_k^{(2)}, \beta_k^{(3)}$ are barycentric coordinates
of $z$ with respect to $a_k^{(1)},a_k^{(2)},$ and $ a_k^{(3)}.$
It is easy to see that $\beta_k^{(1)},\beta_k^{(2)}$ and $\beta_k^{(3)}$ are linear functions
 of the Cartesian coordinates of $a_k^{(1)},a_k^{(2)},$ and $ a_k^{(3)}$ (see e.g. \cite[Section 2]{Floater})
thus
 $\{\beta^{(j)}_k: 1 \le j \le 3\} \subset
C^\infty \Bigl(T\left(a_k^{(1)},a_k^{(2)},a_k^{(3)}\right), [0,1]\Bigr)$
and $\sum_{j=1}^3\beta_k^{(j)}(z)=1$ for all
$z \in T\left(a_k^{(1)},a_k^{(2)}, a_k^{(3)}\right)$.

Define now the functions $u_n:\Omega\to H, n \in \mathbb N,$ by
$$
u_n(\omega)=\sum_{k=1}^\infty \sqrt{f_k(\omega)}\left(\sqrt{\beta_k^{(1)}\bigl(d(\omega)\bigr)}v_{n,k}^{(1)}+\sqrt{\beta_k^{(2)}\bigl(d(\omega)\bigr)}v_{n,k}^{(2)}+\sqrt{\beta_k^{(3)}\bigl(d(\omega)\bigr)}v_{n,k}^{(3)}\right).
$$
By construction, $u_n \in C^r(\Omega)$ for every $n \in \mathbb N.$
Moreover, using \eqref{orthog}, we infer that (ii) holds. Since
$$
\|u_n(\omega)\|^2=\sum_{k=1}^\infty f_k(\omega)\left(\beta_k^{(1)}\bigl(d(\omega)\bigr)+\beta_k^{(2)}\bigl(d(\omega)\bigr)+\beta_k^{(3)}\bigl(d(\omega)\bigr)\right)=
\sum_{k=1}^\infty f_k(\omega)=1
$$
for all $n\in\mathbb N$ and $\omega \in \Omega,$ condition (i) follows.
Moreover, in view of \eqref{ak}, we have
\begin{align*}
&\langle Au_n(\omega),u_n(\omega)\rangle\\
=&
\sum_{k=1}^\infty f_k(\omega)\left(\beta_k^{(1)}\bigl(d(\omega)\bigr)\langle Av_{n,k}^{(1)},v_{n,k}^{(1)}\rangle+\beta_k^{(2)}\bigl(d(\omega)\bigr)\langle Av_{n,k}^{(2)},v_{n,k}^{(2)}\rangle+
\beta_k^{(3)}\bigl(d(\omega)\bigr)\langle Av_{n,k}^{(3)},v_{n,k}^{(3)}\rangle
\right)\\
=&
\sum_{k=1}^\infty f_k(\omega)
\left(\beta_k^{(1)}(d(\omega))a_k^{(1)}+\beta_k^{(2)}(d(\omega))a_k^{(2)}+\beta_k^{(3)}(d(\omega))a_k^{(3)}\right)\\
=&
\sum_{k=1}^\infty f_k(\omega)d(\omega)
=d(\omega)
\end{align*}
for all $n\in\mathbb N$ and $\omega \in \Omega$. Thus, (iii) is satisfied as well.

Let $\omega\in \Omega$.
If $X$ is a Hilbert space with the orthonormal basis $(e_n)_{n=1}^\infty,$
then setting
\[
S(\omega) e_n= u_n(\omega), \qquad \omega \in \Omega,
\]
we obtain an isometry-valued function $S$ on $\Omega$ such that
\[
S^*(\omega) A S(\omega)=d(\omega) I, \qquad \omega \in \Omega.
\]

Finally, let us show that $S$ has the required smoothness. Recalling that for $\omega\in \Omega$ there is a neighborhood $\Omega_\omega$
of $\omega$ such that the set
$F_\omega:=\{k: \supp f_k\cap \Omega_\omega \neq \emptyset
\}$ is nonempty and finite, we let for every $n\in\mathbb N:$
$$
M_n:=\bigvee\{v_{n,k}^{(1)}, v_{n,k}^{(2)}, v_{n,k}^{(3)}:\, k\in F_\omega\}.
$$
Then $\dim M_n=3 \, \text{card} \, F_\omega<\infty$ for all $n,$ and the subspaces $M_n, n \in \mathbb N,$ are mutually orthogonal.
Let $U_n:M_1\to M_n$ be the unitary operator defined by $U_nv_{1,k}^{(i)}=v_{n,k}^{(i)},$ for all
$k\in F_\omega$ and $i=1,2,3.$
Then for all $\omega'\in \Omega_\omega$ and $x \in X,$
we have
\[
S(\omega')(x)=
 \sum_{n=1}^{\infty}\langle x, e_n\rangle u_n(\omega')
= \sum_{n=1}^{\infty}\langle x, e_n\rangle U_n u_1(\omega').
\]
Hence, arguing inductively,
for every $k \in \mathbb N, 1 \le k \le r,$
\begin{equation}
S^{(k)}(\omega')(x)= \sum_{n=1}^{\infty}\langle x, e_n\rangle U_n u_1^{(k)}(\omega'), \qquad \omega' \in \Omega,
\end{equation}
with $S^{(r)} \in C(\Omega, L(X, H)),$ i.e. $S \in C^r (\Omega, L(X, H))$ since the choice of $\omega$ and $\omega'$ was arbitrary.

The cases  $r=\infty$ and $r=0$ can be considered completely analogously.
$\hfill \qed$

\begin{remark}\label{analy}
It is easy to see that one cannot replace smoothness of $u$ and $d$ in
Theorem \ref{partition_intro} by their analyticity.
Indeed, if $A \in L(H)$ and, for instance, one has
$\langle A u(z), u(z) \rangle=d(z)$ for $z \in \mathbb D$, with $u$ and $d$
analytic in $\mathbb D$, then differentiating gives
$2\,{\rm Re}\,\langle A u'(z), u(z) \rangle = d'(z)$ for $z \in \mathbb D$.
The open mapping property then yields that $d'$ is constant, so $d$ is
necessarily a linear function.

\end{remark}

\begin{remark}
Along the lines of the proof above, one can show
that if $A\in L(H),$ $r \in \mathbb N \cup \{0, \infty\},$
 and $D:\Omega\to L(X)$ is a $C^r$-function such that $\Int W_e(A)\supset \{z: |z|\le\sup_{\omega\in\Omega} \|D(\omega)\|\},$ then there exists an isometry-valued $C^r$-function $S:\Omega\to L(X,H)$ satisfying
$S^*(\omega)AS(\omega)=D(\omega)$
for all $\omega\in\Omega$. However, this statement is technically more demanding
and, apart from Bourin's theorem, relies on an appropriate modification of Proposition \ref{P8}.
To keep this paper within reasonable limits, its proof will be given elsewhere.
\end{remark}

\section{Diagonals of operator-valued functions}\label{diag_fun}

In this section we will prove Theorem \ref{diagonal_intro}, one of the main results of this paper.
We start with the next auxiliary lemma.

\begin{lemma}\label{numrangel}
Let $A \in C([0,1],L(H))$.
Then there exist mutually orthogonal subspaces $M_k\subset H, k\in\mathbb N,$ such that
\begin{equation}\label{numrange}
W_e\bigl(P_{M_k}A(t){\upharpoonright}_{M_k}\bigr)=W_e(A(t))
\end{equation}
for all $k\in\mathbb N$ and $t\in[0,1]$.
\end{lemma}
\begin{proof}
Let $Q$ be the set of all rational numbers in $[0,1]$.
For each $q\in Q$ choose a countable set $S_q$ dense in $W_e(A(q))$.
Let $S=\left\{(q,l): q\in Q, l\in S_q\right\}$.
Let $f:\mathbb N\to S\times \mathbb N\times\mathbb N$ be a bijection. Let
\[
f_1:\mathbb N\to Q, \, f_2:\mathbb N\to \cup_{q \in Q} S_q, \, f_3, \, f_4:\mathbb N\to\mathbb N
\]
 be the functions satisfying
$$
f(n)=\bigl((f_1(n),f_2(n)),f_3(n),f_4(n)\bigr)
$$
for all $n\in\mathbb N$. Using Lemma \ref{codim}, (a), construct inductively mutually orthogonal unit vectors $x_n\in H, n \in \mathbb N,$ such that
$$
\left|\left\langle A(f_1(n))x_n,x_n\right \rangle-f_2(n)\right|<\frac{1}{f_3(n)}
$$
for all $n \in \mathbb N.$ For every $k \in \mathbb N,$ set $M_k=\bigvee\{x_n: f_4(n)=k\}$. Clearly the subspaces $M_k, k \in \mathbb N,$ are infinite-dimensional and mutually orthogonal. Fix $k\in\mathbb N$, $q\in Q$ and $l\in S_q$.
For every $j\in\mathbb N$ there exists $n_j$ with
$f_1(n_j)=q$, $f_2(n_j)=l$, $f_3(n_j)=j$ and $f_4(n_j)=k$. So
$$
|\langle A(q)x_{n_j},x_{n_j}\rangle-l|<j^{-1}, \qquad j \ge 1,
$$
and the vectors $(x_{n_j})_{j=1}^{\infty}$ form an orthonormal sequence.
Therefore, $l\in W_e(P_{M_k}A(q){\upharpoonright}_{M_k})$. Since $S_q$ is dense in $W_e(A(q)),$
we have $W_e(A(q))=W_e(P_{M_k}A(q){\upharpoonright}_{M_k}).$ Hence, as $Q$ is dense in $[0,1]$ and the map
$A \to W_e(A)$ is continuous (see Section \ref{prelim}),
we infer that \eqref{numrange} holds.
\end{proof}

We proceed
 with the next version of  Proposition \ref{P4}.  Here, identifying $v(t)$ with rank one isometries $V(t)$, we do not assume that $\dim Z(V)<\infty$ as in Proposition \ref{P4}, which causes some technical complications.
Following the strategy of the proof of Theorem \ref{T9_intro}, this statement will be needed to prepare the inductive step in the proof of Theorem \ref{diagonal_intro}.

It will be convenient to denote by $S_H$ the unit sphere in $H$, i.e. $S_H=\{x\in H: \|x\|=1\}$.

\begin{proposition}\label{6.1}
Let $A\in C([0,1], L(H)),$
and let $d \in C([0,1])$
be such that $d(t)\in {\rm Int}\, W_e(A(t))$ for all $t\in[0,1]$. Let
\[\theta=\min\left\{\dist (d(t),\partial W_e(A(t))): 0\le t\le 1\right\}.\]
Given $\e\in (0,\theta)$ suppose that $v \in C([0,1],  S_H)$ satisfies
$$
\bigl |\langle A(t)v(t),v(t)\rangle-d(t)\bigr|\le \e, \qquad t \in [0,1].
$$
Then for every $\e'\in (0,\e)$ and every
subspace $M\subset H$ such that
\[
W_e(P_MA(t){\upharpoonright}_M)=W_e(A(t)), \qquad t\in [0,1],
\]
there exists $\widetilde v \in C([0,1], H)$ satisfying for all $t \in [0,1],$
\begin{itemize}
\item [(i)] $\bigl| \langle A(t)\widetilde v(t),\widetilde v(t)\rangle-d(t)\bigr|\le \e',$
\item [(ii)] $\|\widetilde v(t)-v(t)\|\le 3\sqrt{\frac{\e}{\theta}},$\\
\item [(iii)] $\widetilde v(t)\in \left(v(t)\vee M\right).$
\end{itemize}
\end{proposition}

\begin{proof}
Without loss of generality we may assume that $d\equiv 0$ (if not, then replace $A$ by $A-d$).
We can also assume that $\|A(t)\|\le 1$ for all $t\in [0,1]$ and
$\e'\le 4\sqrt{\frac{\e}{\theta}}$.

The set $\left\{v(t), A(t)v(t), A^*(t)v(t):0\le t\le 1\right\}$ is compact.
So by Lemma \ref{comp}, there exists a subspace $\widetilde M\subset M$ such that $\dim M/\widetilde M<\infty$, and for all $m\in\widetilde M$ with $\|m\|=1$ and all $t\in[0,1]$ we have
\begin{align*}
|\langle v(t),m\rangle|\le& \e'/8,\\
|\langle A(t)v(t),m\rangle|\le& \e'/8,\\
|\langle A^*(t)v(t),m\rangle|\le& \e'/8.
\end{align*}

Let
$$a=\e \theta^{-1} \qquad \text{and} \qquad  \widetilde d(t):=a^{-1}(a-1) \left\langle A(t)v(t),v(t)\right \rangle, \qquad t \in [0,1].
$$
Clearly $\widetilde d \in C([0,1])$
and
\[
|\widetilde d(t)|\le\frac{(1-a)\e}{a}<\frac{\e}{a}=\theta, \qquad t\in [0,1].
\]

By Lemma \ref{L2}, there exists $r \in C([0,1], S_{\widetilde M})$ such that for all $t \in [0,1],$
$$
\bigl|\bigl\langle A(t)r(t),r(t)\bigr\rangle-\widetilde d(t)\bigr|\le\frac{\e'}{8}.
$$

For every $t \in [0,1]$ let
\[
v'(t)=\sqrt{1-a}\,v(t)+\sqrt{a}\,r(t).
\]
Then $v' \in C([0,1], H)$
and for all $t \in [0,1],$
\begin{align*}
&\bigl|\langle A(t)v'(t),v'(t)\rangle\bigr|\\
=&
\Bigl|(1-a) \bigl\langle A(t)v(t),v(t)\rangle +a \langle A(t)r(t),r(t)\bigr\rangle+
2\sqrt{a(1-a)}\Re\bigl\langle A(t)v(t),r(t)\bigr\rangle\Bigr|\\
\le&
\Bigl| (1-a)\langle A(t)v(t),v(t)\rangle+a\widetilde d(t)\Bigr|+\frac{\e'}{8}+\sqrt{a(1-a)}\cdot \frac{2\e'}{8}
\le\frac{\e'}{4},
\end{align*}
since $\sqrt{a(1-a)}\le 1/2$.
Moreover, for all $t,t' \in [0,1],$
\[
\|v'(t)-v(t)\|\le 1-\sqrt{1-a}+\sqrt{a}\le 2\sqrt{a}\le 2\sqrt{\frac{\e}{\theta}}.
\]

Fix $t\in [0,1].$ We have
\begin{align*}
\|v'(t)\|^2=&
(1-a)\|v(t)\|^2+a\|r(t)\|^2+2\sqrt{a(1-a)}\Re\langle v(t),r(t)\rangle\\
=& 1+2\sqrt{a(1-a)}\Re\langle v(t),r(t)\rangle.
\end{align*}
So
\[
\bigl|1-\|v'(t)\|\bigr|\le \bigl|1-\|v'(t)\|^2\bigr|\le\frac{\e'}{8}.
\]

Let $\widetilde v(t)=\frac{v'(t)}{\|v'(t)\|}$. Then
$\widetilde v(t)\in S_H$ and
$$
\|\widetilde v(t)-v'(t)\|\le \frac{1}{1-\e'/8}-1\le \e'/4.
$$
Hence
$$
\|\widetilde v(t)-v(t)\|\le
\|\widetilde v(t)-v'(t)\|+\|v'(t)-v(t)\|\le \e'/4+2\sqrt{\frac{\e}{\theta}}\le 3\sqrt{\frac{\e}{\theta}}
$$
and
\begin{align*}
\left|\left\langle A(t)\widetilde v(t),\widetilde v(t)\right\rangle\right|
\le&
\left|\left \langle A(t)(\widetilde v(t)-v'(t)),\widetilde v(t)\right\rangle\right|\\
+&
\left|\left\langle A(t)v'(t),\widetilde v(t)-v'(t)\right\rangle\right|+
\left|\left \langle A(t)v'(t),v'(t)\right\rangle\right|\\
\le&
\frac{\e'}{4}+\frac{\e'}{2}+\frac{\e'}{4}=\e',
\end{align*}
so that (ii) and (i) hold.
Moreover, by construction, it is direct that $\widetilde v(t)\in \left(M\vee v(t)\right)$ for all $t\in[0,1],$ and (iii) holds as well.
\end{proof}

The next statement will complement Theorem \ref{T5} in realizing an inductive step
in the proof of Theorem \ref{diagonal_intro}. At the same time, its proof
also relies on an inductive argument, which is made possible by Proposition \ref{6.1}.

\begin{theorem}\label{T12}
Let $A \in C([0,1], L(H)),$  $d \in C([0,1]),$ and $v \in C([0,1], S_H)$ satisfy the assumptions of Proposition \ref{6.1}.
If $M\subset H$ is a subspace such that $W_e(P_MA(t){\upharpoonright}_M)=W_e(A(t)),$ $t \in [0,1],$
then there exists $u \in C([0,1], H)$ such that for all $t \in [0,1],$
\begin{itemize}
\item [(i)]
$\langle A(t)u(t),u(t)\rangle=d(t), \quad \|u(t)\|=1;$
\item [(ii)] $\|u(t)-v(t)\|\le  12\sqrt{\frac{\e}{\theta}};$
\item [(iii)] $u(t)\in \left(v(t)\vee M\right);$
\end{itemize}
\end{theorem}
\begin{proof}
Without loss of generality we may assume that $d\equiv 0$.

Setting $v_0=v,$ we use Proposition \ref{6.1} to construct inductively   continuous mappings $v_n:[0,1]\to S_H, n \in \mathbb N,$ such that
\[
\bigl|\langle A(t)v_n(t),v_n(t)\rangle\bigr|\le\e 2^{-n},
\qquad
\|v_{n+1}(t)-v_n(t)\|\le 3\sqrt{\frac{\e}{\theta 2^n}}
\]
and $v_n(t)\in \left(v(t)\vee M\right)$ for all $n \in \mathbb N$ and $t \in [0,1]$.

Let
\[
u(t)=\lim_{n\to\infty}v_n(t), \qquad t \in [0,1].
\]
Then, for every $t \in [0,1],$
$\langle A(t)u(t),u(t)\rangle=0$ and
\begin{align*}
\|u(t)-v(t)\|\le&\sum_{n=0}^\infty\|v_{n+1}(t)-v_n(t)\|\le\sum_{n=0}^\infty 3\sqrt{\frac{\e}{\theta 2^n}}\\
=&3\sqrt{\frac{\e}{\theta}}\cdot\sum_{n=0}^\infty 2^{-n/2}\le
12\sqrt{\frac{\e}{\theta}}.
\end{align*}
Moreover, by construction, $u(t)\in \left(M\vee v(t)\right)$ for all $t\in[0,1]$.
\end{proof}

Now we are ready to prove Theorem \ref{diagonal_intro}
via an inductive argument based
on  Lemma \ref{numrangel} and Theorems \ref{T12} and \ref{T5}.
\medskip

\emph{Proof of Theorem \ref{diagonal_intro}} \, \,
Consider $A \in C([0,1], L(H))$ and $d_n\in C([0,1]), n\in\mathbb N,$ satisfying $d_n(t)\in\Int W_e(A(t))$ for all $n\in \mathbb N$ and $t\in [0,1]$
and
\[
\theta:=\inf\{\dist(d_n(t),\partial W_e(A(t))): n \in \mathbb N, t \in [0,1]\}>0.
\]
We will construct
 $v_n\in C([0,1], H), n\in\mathbb N,$ such that $(v_n(t))_{n=1}^\infty$ is an orthonormal basis in $H$ for every $t\in[0,1],$
and
$$
\langle A(t)v_n(t),v_n(t)\rangle= d_n(t), \qquad n \in \mathbb N, \quad t \in [0,1].
$$
Moreover, our proof will reveal that the equicontinuity of $\{d_n: n \ge 1\}$
implies
the equicontinuity of $\{ v_n: n \ge 1\}.$

Without loss of generality we may assume that $\|A(t)\|\le 1$ for all $t\in [0,1]$.
Fix a sequence $(y_m)_{m=1}^\infty$ of unit vectors in $H$ such that $\bigvee_{m=1}^\infty y_m=H$.
Note that each $n\in\mathbb N$ can be written uniquely as
\[n=2^{r(n)-1}(2s(n)-1)\]
 with integers $r(n), s(n)\in\mathbb N.$
Fix $\eta \in (0,\sqrt \theta/2\sqrt 2),$ and
let
\[
\eta_k=\frac{\eta}{2^{r(k)/2}{s(k)^{1/2}}}, \qquad k \in \mathbb N.
\]

The functions  $v_k:[0,1]\to S_H, k \ge 1,$ will be constructed inductively.

Using Lemma \ref{numrangel}, find mutually orthogonal subspaces  $M_k, k\in\mathbb N,$ of $H$ such that $W_e(P_{M_k}A(t){\upharpoonright}_{M_k})=W_e(A(t))$ for all $k\in\mathbb N$ and $t\in[0,1]$.

Fix $k\in\mathbb N,$ and suppose that we have already constructed continuous functions $v_1,
\dots,v_{k-1}:[0,1]\to S_H$ such that
\begin{itemize}
\item [(1)] the vectors $v_1(t),\dots,v_{k-1}(t)$ are mutually orthogonal for each $t \in [0,1]$;

\item[(2)] $\langle A(t)v_j(t),v_j(t)\rangle= d_j(t)$ for all $j\le k-1, t\in[0,1]$;

\item [(3)] if $j\le k-1, t \in [0,1],$ then $v_j(t)\in\bigvee_{i=1}^j \left(M_{i}\vee y_{r(i)}\right)$;

\item [(4)]  if $H_j(t)=\bigvee_{i=1}^j v_i(t)$, then  for all $j \le k-1$ and $t \in [0,1],$
\[
\dist^2(y_{r(j)},H_j(t))\le \prod_{i=1}^{s(j)}\left(1-\eta^2_{2^{ r(j)-1 }(2i-1)}\right).
\]
\end{itemize}
Moreover,
if $\{ d_n: n=1,2,\dots\}$ are equicontinuous with a modulus of continuity $\delta$ and $\delta'$ is the modulus of continuity constructed in Theorem \ref{T5}, then we also assume that

(5) there exist functions $u_j:[0,1]\to S_H,$ $1 \le j \le k-1,$ admitting the modulus of continuity $\delta'$ such that
$$
\|u_j(t)-v_j(t)\|\le 27\sqrt{2}\eta_k{\theta}^{-1/2}.
$$
for all $t \in [0,1]$ and $1 \le j \le k-1.$

Observing that if $k=1$, then (1)--(5) are formally satisfied,
we construct a function $v_k$ in the following way.

For each $t\in[0,1]$ write
\[
y_k(t)=(I-P_{H_{k-1}(t)})y_{r(k)}.
\]
Let
\[
b_k=\prod_{i=1}^{s(k)}(1-\eta_{2^{(r(k)-1)}(2i-1)})>0 \qquad \text{and} \qquad b_k(t)=\max\{b_k,\|y_k(t)\|\}
\]
for all $t \in [0,1].$
Finally, choose $\delta_k>0$ such that
\[
12\sqrt{\frac{2\delta_k}{\theta}}\le\eta_k/2.
\]

Note that the set
\[
\left\{\frac{A(t)y_k(t)}{b_k(t)}, \, \frac{A^*(t)y_k(t)}{b_k(t)}: \, t\in [0,1]\right\}
\]
is compact. Therefore by Lemma \ref{comp} there exists a subspace
$\widetilde M_k\subset M_k$ such that
\[
\dim M_k/\widetilde M_k<\infty, \qquad
\widetilde M_k\perp\{y_{r(j)}:1\le j\le k\},
\]
and, moreover,
\[
\left|\left\langle \frac{A(t)y_k(t)}{b_k(t)},m\right\rangle\right|\le\delta_k, \qquad
\left|\left\langle \frac{A^*(t)y_k(t)}{b_k(t)},m\right\rangle\right|\le\delta_k,
\]
for all $m\in \widetilde M_k$ with $\|m\|=1$ and $t\in[0,1]$.

By Theorem \ref{T5}, there exists a continuous function $u_k:[0,1]\to S_{\widetilde M_k}$ such that
$\langle A(t)u_k(t),u_k(t)\rangle=d_k(t)$ for all $t\in [0,1]$. Observe that if
$\{d_n: n\ge 1\}$ are equicontinuous and admit the same modulus of continuity $\delta$,
then we may assume that $u_k$ admit $\delta'$, where $\delta'$ is the modulus of continuity constructed in Theorem \ref{T5}.

Let
$$
\widetilde d_k(t)=\frac{b_k^2(t)}{b^2_k(t)-\eta_k^2\|y_k(t)\|^2}d_k(t)-\frac{\eta_k^2 \langle A(t)y_k(t),y_k(t)\rangle}{b_k^2(t)-\eta_k^2\|y_k(t)\|^2}.
$$
We have
\begin{align*}
\bigl|\widetilde d_k(t)-d_k(t)\bigr|\le&
\left|1-\frac{b_k^2(t)}{b^2_k(t)-\eta_k^2\|y_k(t)\|^2}\right|+\left|\frac{\eta_k^2 \langle A(t)y_k(t),y_k(t)\rangle}{b_k^2(t)-\eta_k^2\|y_k(t)\|^2}\right|\\
\le&\frac{2\eta_k^2\|y_k(t)\|^2}{b_k^2(t)-\eta_k^2\|y_k(t)\|^2}\le\frac{2\eta_k^2}{1-\eta_k^2}\\
\le&
4\eta_k^2\le 4\eta^2
\le \theta/2.
\end{align*}
So
\[
\left|\langle A(t)u_k(t),u_k(t)\rangle-\widetilde d_k(t)\right|\le 4\eta_k^2\le\theta/2\]
and
$$
\dist(\widetilde d_k(t), \partial W_e(A(t)))\ge\theta/2
$$
for all $t\in [0,1]$.

By Theorem \ref{T12}, there exists a continuous function $u_k^{(1)}:[0,1]\to S_{\widetilde M_k}$ such that
$$
\bigl\langle A(t) u_k^{(1)}(t), u_k^{(1)}(t)\bigr\rangle=\widetilde d_k(t)
$$
for all $t\in [0,1]$ and
$$
\|u_k^{(1)}(t)-u_k(t)\|\le \frac{24\sqrt{2}\eta_k}{\theta^{1/2}}.
$$

Define
$$
u_k^{(2)}(t)=\sqrt{1-\frac{\eta_k^2\|y_k(t)\|^2}{b_k^2(t)}} u_k^{(1)}(t)+\frac{\eta_k y_k(t)}{b_k(t)}, \qquad t \in [0,1].
$$
Since $u_k^{(1)}(t)\in\widetilde M_k$ and
\[
y_k(t)\in\left(\bigvee_{j=1}^{k-1}u_j(t)\vee y_{r(k)}\right)\subset \bigvee_{j=1}^{k-1}M_j\vee\bigvee_{j=1}^{k}y_{r(j)},
\]
 we have $y_k(t)\perp u_k^{(1)}(t)$ and
$\|u_k^{(2)}(t)\|=1$ for all $t \in [0,1]$.
In addition, in view of
\[
u_k^{(2)}(t)\in \left(\widetilde M_k \vee y_{k}(t)\right),
\]
we also have $u_k^{(2)}(t)\perp H_{k-1}(t)$ for all $t\in[0,1]$.

Thus,
\begin{align*}
&\left|\Bigl\langle A(t) u_k^{(2)}(t),u_k^{(2)}(t)\Bigr\rangle-d_k(t)\right|\\
\le&
\Bigl|\Bigl(1-\frac{\eta_k^2\|y_k(t)\|^2}{b^2_k(t)}\Bigr)\langle A(t) u_k^{(1)}(t), u_k^{(1)}(t)\rangle +\frac{\eta_k^2}{b^2_k(t)}\langle A(t)y_k(t),y_k(t)\rangle-d_k(t)\Bigr|\\
+&\frac{\eta_k}{b_k(t)} \cdot
\sqrt{\frac{b_k^2(t)-\eta_k^2\|y_k(t)\|^2}{b_k^2(t)}}\Bigl|\langle A(t)y_k(t),u_k^{(1)}(t)\rangle+\langle A(t) u_k^{(1)}(t),y_k(t)\rangle\Bigr|\\
\le&
\Bigl|\frac{b_k^2(t)-\eta_k^2\|y_k(t)\|^2}{b_k^2(t)} \widetilde d_k(t) +\frac{\eta_k^2}{b_k^2(t)}\langle A(t)y_k(t),y_k(t)\rangle-d_k(t)\Bigr|+2\eta_k\delta_k\\
\le& 2\delta_k.
\end{align*}
Moreover,
$$
\|u_k^{(2)}(t)-u_k^{(1)}(t)\|\le
1-\sqrt{1-\frac{\eta_k^2\|y_k(t)\|^2}{b_k^2(t)}}+\frac{\eta_k\|y_k(t)\|}{b_k(t)}\le
\Bigl(1-\sqrt{1-\eta_k^2}\Bigr)+\eta_k\le 2\eta_k.
$$
By Theorem \ref{T12}, there exists a continuous function $v_k:[0,1]\to S_H$ such that
for all $t \in [0,1],$
$$
v_k(t)\in u_k^{(2)}(t)\vee\widetilde M_k\subset \bigvee_{j=1}^{k} \left(M_j\vee y_{r(j)}\right),
$$
$$
\|v_k(t)- u_k^{(2)}(t)\|\le 12\sqrt{\frac{2\delta_k}{\theta}}\le\eta_k/2,
$$
and
$$
\langle A(t)v_k(t),v_k(t)\rangle=d_k(t).
$$
In particular,
\[
v_k(t)\perp v_1(t),\dots,v_{k-1}(t), \quad t \in [0,1].
\]
We conclude that (1)--(3) hold, and it remains to establish (4).

To this aim, observe that if $\|y_k(t)\|< b_k,$ then (4) is satisfied trivially.
If otherwise $\|y_{k}(t)\|\ge b_k,$ then $b_k(t)=\|y_k(t)\|$ and
$$
\dist^2(y_k(t),H_k(t))=
\|y_k(t)\|^2-|\langle y_k(t),v_k(t)\rangle|^2,
$$
where
\begin{align*}
|\langle y_k(t),v_k(t)\rangle|\ge&
|\langle y_k(t),u_k^{(2)}(t)\rangle|-\|y_k(t)\|\cdot\|v_k(t)-u_k^{(2)}(t)\|\\
\ge&
\frac{\eta_k\|y_k(t)\|^2}{b_k(t)}-\|y_k(t)\|\cdot \eta_k/2\ge \frac{\eta_k\|y_k(t)\|}{2}.
\end{align*}
So
$$
\dist^2(y_k(t),H_k(t))\le
\|y_k(t)\|^2\Bigl(1-\frac{\eta_k^2}{4}\Bigr)\le
\prod_{i=1}^{s(k)}\left(1-\frac{\eta^2_{2^{r(k)}(2i-1)}}{4}\right)
$$
by the induction assumption.

As a result of this inductive construction, we obtain the sequence $(v_k)_{k=1}^\infty \subset C([0,1], H)$
 such that $(v_k(t))_{k=1}^\infty$ is an orthonormal system for each $t\in [0,1]$
and \[
\langle A(t)v_k(t),v_k(t)\rangle=d_k(t)
\]
 for all $k\in\mathbb N$ and $t\in [0,1]$.

Moreover, by construction, if the functions
$\{d_k: k \ge 1\}$ are equicontinuous with modulus of continuity $\delta,$ then the functions
 $\{u_k: k \ge 1\}$ are equicontinuous with modulus of continuity $\delta'$ constructed in Theorem \ref{T5}. In addition, for all $t\in[0,1],$ we have
\begin{align*}
\|v_k(t)-u_k(t)\|&\le
\|v_k(t)-u_k^{(2)}(t)\|+\|u_k^{(2)}(t)- u_k^{(1)}(t)\|+\|u_k^{(1)}(t)-u_k(t)\|\cr
&\le\frac{\eta_k}{2}+2\eta_k+\frac{24\sqrt{2}\eta_k}{\theta^{1/2}}\le 27\sqrt{2}\eta_k\theta^{-1/2}.
\end{align*}
So $\sup\{\|v_k(t)-u_k(t)\|: t\in[0,1]\}\to 0$ as $k\to\infty,$ hence $\{ v_k: k \ge 1\}$ are equicontinuous with modulus of continuity $\delta'.$

It remains to prove that the sequence $(v_k(t))_{k=1}^\infty$ is an orthonormal basis in $H$ for every $t \in [0,1]$. To show this property, let $m\in\mathbb N$ and $t\in[0,1]$ be fixed. Then
$$
\dist^2(y_m,\bigvee_{j=1}^\infty v_j(t))=
\lim_{s\to\infty}\dist^2(y_m,H_{2^{m-1}(2s-1)}(t))
\le
\prod_{i=1}^\infty\left(1-\frac{\eta^2_{2^{m-1}(2i-1)}}{4}\right).
$$
Hence
\begin{align*}
\ln\dist^2(y_m,\bigvee_{j=1}^\infty v_j(t))&\le
\sum_{i=1}^\infty\ln\left(1-\frac{\eta^2_{2^{m-1}(2i-1)}}{4}\right)\\
&\le
-\sum_{i=1}^\infty\frac{\eta^2_{2^{m-1}(2i-1)}}{4}=
-\sum_{i=1}^\infty\frac{\eta^2}{i2^{m+2}}=-\infty,
\end{align*}
so that $y_m\in\bigvee_{k=1}^\infty v_k(t)$ for each $t \in [0,1]$.
Since $\bigvee_{m\in\mathbb N} y_m=H$, the property follows,
and thus the proof is complete.

$\hfill \qed$

\section{Compressions and diagonals of operator-valued functions with selfadjoint values}\label{self}

In this section, we show how to obtain  results analogous to those in Sections \ref{compr}
and \ref{diag_fun} for  selfadjoint $A.$
Here we formulate the results using more explicit spectral terms, which make them particularly
attractive.
Recall that for a selfadjoint operator $A\in L(H)$ we have $\overline{W(A)}=\conv \sigma(A)$ and $W_e(A)={\rm conv}\,\sigma_e(A),$ where $\sigma_e(A)$ denotes the essential spectrum of $A$
and all of the sets are subsets of $\mathbb R$. Thus,  $\Int_{\mathbb R} W_e(A)=\Int_{\mathbb R}\conv \sigma_e(A),$ where $\Int_{\mathbb R}$ stands for the interior relative to $\mathbb R.$
Lemma \ref{codim} remains valid after replacing $\Int W_e$ by  $\Int_{\mathbb R} W_e.$
We start with obtaining a counterpart of Theorem \ref{T9_intro}
for selfadjoint $A$ and $D.$
In the selfadjoint case, the factorization approach from Section \ref{compr}
simplifies and requires just two factors to separate the two time-dependent sides of
\eqref{compress}.

The next results are analogues of Lemma \ref{L2}, Proposition \ref{P4} and Proposition \ref{P6},
respectively, and have similar, and in fact more direct, proofs
(using relative interiors instead of interiors).
The statements are formulated for selfadjoint $A(t)$, $t \in [0,1]$, with $\Int_{\mathbb R}\sigma_e(A(t))\supset [a,b]$ for some $a,b \in \mathbb R$, $a<b$, and all $t \in [0,1]$.
To deduce them from the corresponding statements of the preceding sections, it suffices to pass to $A(t)-(a+b)/2$ and $D(t)-(a+b)/2$ in place of $A(t)$ and $D(t)$ for $t \in [0,1]$.
For similar situations with general $A$ and selfadjoint $A$ considered
in a time-independent setting, see e.g.\ \cite{MT19}
and \cite{LW20}.

\begin{lemma}
Let $A \in C([0,1], L(H))$ satisfy $A^*(t)=A(t), t \in [0,1],$ and assume that there exist $a, b \in \mathbb R, a<b,$ such that $[a,b]\subset
\Int_{\mathbb R} {\rm conv}\,\sigma_e(A(t))$
for all $t \in [0,1]$. For a Hilbert space $X$ with $\dim X<\infty$, let $D \in C([0,1], L(X))$
be such that
$D^*(t)=D(t)$ and $\sigma(D(t))\subset [a,b]$ for all $t \in [0,1]$. Then for every subspace $L\subset H$,
$\codim L<\infty,$ and every $\e>0,$
there exists $V \in C([0,1], L(X,H))$ such that for all $t\in [0,1],$
\begin{itemize}
\item[(i)] $V(t)$ is an isometry;
\item[(ii)] $\dim Z(V)<\infty$ and $Z(V)\subset L$;
\item[(iii)] $\|V^*(t)A(t)V(t)-D(t)\|\le\e$;
\item[(iv)] $\Lip (V)<\infty,$ and $\Lip(V)$ depends only on $\e$ and the moduli of continuity $\delta_A$ and $\delta_D$.
\end{itemize}
\end{lemma}

\begin{proposition}
Let $A \in C([0,1], L(H))$ satisfy $A^*(t)=A(t), t \in [0,1],$ and $[a,b] \subset \Int_{\mathbb R} {\rm conv}\,\sigma_e(A(t))$ for some $a, b \in \mathbb R, a<b,$ and all $t \in [0,1].$
Let $X$ be a finite-dimensional Hilbert space and let $D \in C([0,1], L(X))$ be such that
$D^*(t)=D(t)$
 and $\sigma(D(t))\subset [a,b]$ for all $t\in [0,1]$. Set
$$
\theta=\min_{t\in [0,1]}\min\{\|b-D(t)\|,\|D(t)-a\|\}.
$$
Given $\e, l >0,$ and a subspace $L \subset H, {\rm codim}\, L<\infty,$ suppose that
$V \in C([0,1], L(X,H))$ is isometry-valued,
$\dim Z(V)<\infty$, $Z(V) \subset L,$ $\Lip(V)\le l,$ and
$$
\|V^*(t)A(t)V(t)-D(t)\|\le\e, \qquad t \in [0,1].
$$
Then for every $\e'>0$ there exist $\tilde l$ depending only on $\e', l$ and
the moduli of continuity $\delta_A$ and $\delta_D$ and
$\tilde V \in C([0,1], L(X,H))$ such that
\begin{itemize}
\item [(i)] $\tilde V$ is isometry-valued,
\item [(ii)] $Z(\tilde V)\subset L$, $\dim Z(\tilde V)<\infty$ and $\Lip(\tilde V)\le\tilde l,$
\item [(iii)]$
\|\tilde V^*(t)A(t)\tilde V(t)-D(t)\|\le\e', \, t \in [0,1],
$
\item [(iv)]
$
\|\tilde V(t)-V(t)\|\le \sqrt{\frac{2\e}{\theta}}, \, t \in [0,1].
$
\end{itemize}
\end{proposition}

\begin{lemma}\label{selfd}
Let $A \in C([0,1], L(H))$ be such that $A^*(t)=A(t), t \in [0,1],$ and there are $a, b \in \mathbb R, a<b,$ satisfying $[a,b]\subset \Int_{\mathbb R} \conv \sigma_e(A(t))$
for all $t\in [0,1].$ Assume that $X$ is a separable Hilbert space, $X=X_1\oplus X_2$ with $\dim X_1=\dim X_2=\infty$, and $D\in L(X)$ is given by
$D=a'I\oplus b'I,$ $a<a'<b'<b$. Then there exists $V \in C([0,1], L(X,H))$ such that $V$ is isometry-valued,
and
$$
V^*(t)A(t)V(t)=D, \qquad t \in [0,1].
$$
\end{lemma}

We will need the next auxiliary statement serving a role similar to Proposition \ref{P8}
in the particular setting of selfadjoint operators.
It is proved in a generality greater than needed for our reasoning,
but its proof does not deviate much from the proof of a simpler version.

\begin{lemma}\label{selfdd}
Let $A \in L(H)$ be such that $A=A^*$ and  $[a,b] \subset \Int_\mathbb R \conv \sigma_e (A)$
for some $a,b \in \mathbb R, a<b.$
 If $Z$ is a separable Hilbert space and  $D \in C([0,1], L(Z))$
satisfies $D^*(t)=D(t)$ and $\sigma (D(t))\subset [a, b]$
for all
$t \in [0,1]$,
then there exists an isometry-valued $W \in C([0,1], L(Z,H))$ such that
\begin{equation}
W^*(t)A W(t)=D(t), \qquad t \in [0,1].
\end{equation}
\end{lemma}

\begin{proof}
Note first that there exist orthogonal infinite-dimensional subspaces $H_a, H_b\subset H$ such that
$$
P_{H_a\oplus H_b}A{\upharpoonright}_{H_a\oplus H_b} =
\begin{pmatrix}
aI&0\cr 0&bI
\end{pmatrix}.
$$
Indeed, using Lemma \ref{codim} we construct inductively orthonormal vectors $e_1,f_1,$ $e_2,f_2,e_3,\dots$ in $H$ such that for all $j\in\mathbb N$, $\langle Ae_j,e_j\rangle=a$, $\langle Af_j,f_j\rangle=b$,
$$
e_j\perp\{e_1,\dots,e_{j-1}, f_1,\dots,f_{j-1},Ae_1,\dots,Ae_{j-1},Af_1,\dots, Af_{j-1}\},
$$
and
$$
f_j\perp\{e_1,\dots,e_{j}, f_1,\dots,f_{j-1},Ae_1,\dots,Ae_{j},Af_1,\dots, Af_{j-1}\}.
$$
Set $H_a=\bigvee\{e_j: j\in\mathbb N\}$ and $H_b=\bigvee\{f_j: j\in\mathbb N\}$.
Then the compression $P_{H_a\oplus H_b}A {\upharpoonright}_{H_a\oplus H_b}$ has the required form.

Choose now isometries $W_a:Z\to H_a$ and $W_b:Z\to H_b,$ and for every $t \in [0,1]$
define
$W(t):Z\to H_a\oplus H_b$ by
\begin{align*}
W(t)x=\frac{W_a(b-D(t))^{1/2}x}{(b-a)^{1/2}}\oplus\frac{W_b(D(t)-a)^{1/2}x}{(b-a)^{1/2}}, \qquad x \in Z.
\end{align*}
Then for all $x\in Z$ and $t \in [0,1],$ we have
\begin{align*}
\|W(t)x\|^2=&\frac{\|(b-D(t))^{1/2}x\|^2}{b-a}+\frac{\|(D(t)-a)^{1/2}x\|^2}{b-a}\\
=&
\frac{\langle(b-D(t))x,x\rangle}{b-a}+\frac{\langle(D(t)-a)x,x\rangle}{b-a}\\
=&\frac{(b-a)\|x\|^2}{b-a}=\|x\|^2,
\end{align*}
i.e. $W(t)$ is an isometry.
Moreover, an argument similar to the one for $\Delta$ in the proof of Proposition \ref{P8}
shows that $W \in C([0,1], L(Z, H)).$

Finally, for every $x\in Z,$
\begin{align*}
\langle W^*(t)AW(t)x,x\rangle
=\langle AW(t)x,W(t)x\rangle
=&
\frac{a}{b-a}\langle(b-D(t))x,x\rangle\\
+& \frac{b}{b-a}\langle(D(t)-a)x,x\rangle=
\langle D(t)x,x\rangle,
\end{align*}
so that $W^*A W =D$ for all $t \in [0,1],$ as required.
\end{proof}

The next statement is a counterpart of Theorem \ref{T9_intro}
in the case of operator functions $A$ with selfadjoint values.

\begin{theorem}\label{t2self}
Let $A \in C([0,1], L(H))$ be such that  $A^*(t)=A(t), t \in [0,1],$ and $[a,b]\subset \Int_{\mathbb R} \conv\sigma_e(A(t)),$
for some $a, b \in \mathbb R, a<b,$ and all $t\in [0,1].$
Assume that $Z$ is a separable Hilbert space,
and $D\in C([0,1], L(Z))$ satisfies $D^*(t)=D(t)$ and $\sigma(D(t))\subset (a,b).$
Then there exists $S\in C([0,1], L(Z,H))$ such that $S$ is isometry-valued and
$$
S^*(t)A(t)S(t)=D(t), \qquad t \in [0,1].
$$
\end{theorem}
\begin{proof}
Choose $a<a'<b'<b$ such that $\sigma (D(t)) \subset [a', b']$ for all $t \in [0,1].$
Let $X_1,X_2$ be infinite-dimensional Hilbert spaces, $X=X_1\oplus X_2$ and
$D_{a',b'}=
\begin{pmatrix}a'I&0\cr0&b'I
\end{pmatrix}.$
Using Lemmas \ref{selfd} and \ref{selfdd}, one
obtains isometry-valued functions $V \in C([0,1], L(X,H))$ and $W \in C([0,1], L(Z, X))$ satisfying
\begin{equation*}\label{selfeq}
V^*(t)A(t)V(t)=D_{a', b'} \qquad \text{and} \qquad  W^*(t) D_{a', b'} W(t)=D(t), \quad t \in [0,1].
\end{equation*}
Now it suffices to set $S(t)=V(t)W(t), t \in [0,1].$
\end{proof}

We finish this section with the following variant of Theorem \ref{diagonal_intro}
for selfadjoint $A,$ where again spectral terms makes the formulation
more transparent.
\begin{theorem}\label{diagonalself}
Let $A \in C([0,1], L(H))$ have selfadjoint values, and $d_n\in C([0,1]), n\in\mathbb N.$
Assume that there are $a, b \in \mathbb R$ such that
\[
d_n(t)\in [a, b] \subset \Int_{\mathbb R} \conv \sigma_e(A(t)), \qquad n\in \mathbb N, \quad t\in [0,1].
\]
Then there exist  $(v_n)_{n=1}^{\infty} \subset C([0,1], H)$ such that
$(v_n(t))_{n=1}^\infty$ is an orthonormal basis in $H$ for every $t\in[0,1],$
and
$$
\langle A(t)v_n(t),v_n(t)\rangle= d_n(t), \qquad n \in \mathbb N, \quad t \in [0,1].
$$
Moreover, the sequence $(v_n)_{n=1}^\infty$ is equicontinuous if the sequence $(d_n)_{n=1}^\infty$ is so.
\end{theorem}
The proof of this result is completely analogous to the proof of Theorem \ref{diagonal_intro},
with substantial simplifications caused by the selfadjointness of $A$.

\section{Acknowledgments}

We are grateful to the referee for a careful reading of the manuscript
and for many helpful suggestions and remarks, which have led to several improvements in the paper.

\end{document}